\documentclass[11pt,fleqn]{article}
\usepackage{geometry}                
\usepackage{graphicx}
\usepackage{amsthm,amsmath}
\usepackage{amssymb,dsfont,mathrsfs}

\numberwithin{equation}{section}
\newtheorem{theorem}{Theorem}[section]
\newtheorem{lemma}{Lemma}[section]
\newtheorem{proposition}{Proposition}[section]
\newtheorem{corollary}{Corollary}[section]
\newtheorem{remark}{Remark}[section]
\newtheorem{assumption}{Assumption}[section]

\newtheorem{conjecture}{Conjecture}[section]
\newtheorem{condition}{Condition}[section]
\def\ba{\boldsymbol{a}}
\def\bb{\boldsymbol{b}}

\def\be{\boldsymbol{e}}

\def\bi{\boldsymbol{i}}

\def\bl{\boldsymbol{l}}
\def\bm{\boldsymbol{m}}

\def\bs{\boldsymbol{s}}

\def\bu{\boldsymbol{u}}
\def\bv{\boldsymbol{v}}

\def\bx{\boldsymbol{x}}
\def\by{\boldsymbol{y}}

\def\bX{\boldsymbol{X}}
\def\bY{\boldsymbol{Y}}

\def\bpi{\boldsymbol{\pi}}
\def\bnu{\boldsymbol{\nu}}
\def\btheta{\boldsymbol{\theta}}

\def\bzero{\mathbf{0}}

\def\scrI{\mathscr{I}}
\def\scrP{\mathscr{P}}
\def\spr{\mbox{\rm spr}}
\def\cp{\mbox{\rm cp}}
\def\diag{\mbox{\rm diag}}

\title{Convergence parameters of nonnegative block tri-diagonal matrices and their application to multi-dimensional  QBD processes}
\author{Toshihisa Ozawa  \\ 
Faculty of Business Administration, Komazawa University \\
1-23-1 Komazawa, Setagaya-ku, Tokyo 154-8525, Japan \\
E-mail: toshi@komazawa-u.ac.jp
}
\date{\today}

\begin{document}

\maketitle

\begin{abstract}
First, we consider a nonnegative homogeneous block tri-diagonal matrix and obtain its convergence parameter, where some results in the field of matrix analytic method are extended to the case where block matrices have countably infinite dimension. 
Second, we apply our results to a multi-dimensional QBD process and obtain lower bounds for the directional asymptotic decay rates of the stationary distribution. 

\smallskip
{\it Keywards}: Nonnegative matrix, convergence parameter, quasi-birth-and-death process, Markov modulated reflecting random walk, stationary tail distribution, asymptotic decay rate 

\end{abstract}

%
%
\section{Introduction} \label{sec:intro}

We develop tools for analyzing asymptotics in discrete-time multi-dimensional quasi-birth-and-death processes (QBD processes for short) and demonstrate their effectiveness. 
A discrete-time multi-dimensional QBD process is an extension of ordinary discrete-time (one-dimensional) QBD process and its typical examples are queueing network models and multiqueue models arising from production systems, information network systems and other service systems. 
Let $\{\bY_n\}=\{(\bX_n,J_n)\}$ be a $d$-dimensional QBD process on the state space $\mathbb{S}_+=\mathbb{Z}_+^d\times S_0$, where $S_0=\{1,2,...,s_0\}$ is the phase space, $\bX_n=(X_{1,n},X_{2,n},...,X_{d,n})$ is the level state and $J_n$ is the phase state. We assume the cardinality of $S_0$, $s_0$, is finite. 
The $d$-dimensional QBD process $\{\bY_n\}$ is also a $d$-dimensional skip-free Markov modulated reflecting random walk (MMRRW for short) with background state $J_n$. 

As pointed out in \cite{Miyazawa11}, there are several approaches to attack an asymptotic problem for the stationary distribution of such a multidimensional process. Here we take a Markov additive approach based on Matrix analytic methods in queueing theory. The research field of matrix analytic method was originated by M.F.\ Neuts and it has been providing many algorithms to compute the stationary distributions and related performance measures for various queueing models (see, for example, \cite{Neuts94,Latouche99}).  
Matrix analytic methods are also used for analyzing asymptotics of the stationary distributions of queueing models including those having two queues (see, for example, \cite{He09,Miyazawa04,Miyazawa09,Miyazawa15,Ozawa13,Ozawa18,Takahashi01}). 
In such literature, the behavior of a queueing model is represented as a QBD process with countably many phase states, whose transition probability matrix $P$ is given in tri-diagonal block form as 
\begin{equation}
P=\begin{pmatrix}
B_0 & B_1 & & & \cr
B_{-1} & A_0 & A_1 & & \cr
& A_{-1} & A_0 & A_1 & \cr
& & \ddots & \ddots & \ddots
\end{pmatrix}, 
\label{eq:P_intro}
\end{equation}
where the dimensions of the block matrices are countably infinite. Like QBD processes having finite phase states, such a QBD process also have the stationary distribution $\bpi=(\bpi_k)$ given in block form as
\begin{equation}
\bpi_k = \bpi_1 R^{k-1},\ k\ge 1, 
\end{equation}
where $R$ is the rate matrix determined by the triplet $\{A_{-1},A_0,A_1\}$ (see, for example, \cite{Tweedie82}). Hence, we can investigate asymptotics of the stationary distribution through the rate matrix $R$ and the following formula is used as a key formula (see, Lemma 2.2 of \cite{He09} and Lemma 2.3 of \cite{Miyazawa09}): 
\begin{equation}
\cp(R) = \sup\{\theta\in\mathbb{R};\, \bx A_*(\theta) \le \bx\ \mbox{for some}\ \bx > \bzero^\top \}, \label{eq:cpR_intro}
\end{equation}
where, for a matrix $A$, $\cp(A)$ is the convergence parameter of $A$, $\bzero$ is a column vector of 0's and $A_*(\theta)$ is given as
\begin{equation}
A_*(\theta) = e^{-\theta} A_{-1} + A_0 + e^{\theta} A_1.
\end{equation}

A $d$-dimensional QBD process $\{\bY_n\}=\{((X_{1,n},X_{2,n},...,X_{d,n}),J_n)\}$ introduced above can also be represented as $d$ kinds of QBD process with countably many phase states, for example, one is $\{\bY_n^{(1)}\}=\{(X_{1,n},(X_{2,n},...,X_{d,n},J_n))\}$, where $X_{1,n}$ is the level state and $(X_{2,n},...,X_{d,n},J_n))$ is the phase state. 
Here we consider $\{\bY^{(1)}\}$ and focus on formula (\ref{eq:cpR_intro}). Inequality $\bx A_*(\theta)\le \bx$ for $\bx>\bzero^\top$ implies that $\cp(A_*(\theta))\ge 1$ and we have 
\begin{equation}
\cp(R) = \sup\{\theta\in\mathbb{R}; \cp(A_*(\theta))^{-1} \le 1\}.
\end{equation}
Hence, in order to obtain $\cp(R)$, it suffices to consider $\cp(A_*(\theta))$ for $\theta\in\mathbb{R}$. 
The phase state of $\{\bY^{(1)}\}$ is given by $(X_{2,n},...,X_{d,n},J_n)$ and hence, for example, by regarding $X_{2,n}$ as the level state and $(X_{3,n},...,X_{d,n},J_n)$ as the phase state, the block matrices $A_{-1}$, $A_0$ and $A_1$ can be represented in block tri-diagonal form. This leads us to the following representation of $A_*(\theta)$:  
\begin{equation}
A_*(\theta)=\begin{pmatrix}
B_{*,0}(\theta) & B_{*,1}(\theta) & & & \cr
B_{*,-1}(\theta) & A_{*,0}(\theta) & A_{*,1}(\theta) & & \cr
& A_{*,-1}(\theta) & A_{*,0}(\theta) & A_{*,1}(\theta) & \cr
& & \ddots & \ddots & \ddots
\end{pmatrix}. 
\label{eq:As_intro}
\end{equation}
This $A_*(\theta)$ is a nonnegative block tri-diagonal matrix, but it may no longer be stochastic or substochastic. Therefore, in order to apply matrix analytic methods to $A_*(\theta)$, we must extend them to nonnegative block tri-diagonal matrices with countably many phase states. 
\textit{For the case in which $B_{*,-1}(\theta)=A_{*,-1}(\theta)$, $B_0(\theta)=A_0(\theta)$ and $B_1(\theta)=A_1(\theta)$, we do it in Section \ref{sec:RandGmatrix}}, where the rate matrix and G-matrix of a general nonnegative block tri-diagonal matrix are introduced and their properties are clarified. 
If $A_{*,-1}(\theta)$, $A_{*,0}(\theta)$ and $A_{*,1}(\theta)$ are also represented in block tri-diagonal form, our approach can recursively be applied. 
We note that References \cite{Kijima93,Li03} discussed the case where a matrix corresponding to $A_*(\theta)$ was a substochastic matrix with finite phase states. Our results are also an extension of their results. 

In Section \ref{sec:modelandresults}, we apply our results to a multi-dimensional QBD process and obtain a lower bound for the asymptotic decay rate of the stationary distribution in each coordinate direction. We also discuss multi-dimensional Markov additive processes arising from the multi-dimensional QBD process. 
In Section \ref{sec:conclusion}, we present a conjecture for the directional asymptotic decay rates of the stationary distribution in multi-dimensional QBD processes as a concluding remark.

\medskip
\textit{Notations.} 
$\mathbb{R}$ is the set of all real numbers and $\mathbb{R}_+$ that of all nonnegative real numbers. $\mathbb{Z}$ is the set of all integers and $\mathbb{Z}_+$ that of all nonnegative integers. $\mathbb{N}$ is the set of all positive integers and, for $n\ge 1$, $\mathbb{N}_n$ is that of positive integers less than or equal to $n$. 
For a matrix $A$, we denote by $[A]_{i,j}$ the $(i,j)$-element of $A$. The transpose of a matrix $A$ is denoted by $A^\top$. The convergence parameter of a nonnegative matrix $A$ with a finite or countable dimension is denoted by $\cp(A)$, i.e., $\cp(A) = \sup\{z\in\mathbb{R}_+; \sum_{n=0}^\infty z^n A^n<\infty \}$. We denote by $\spr(A)$ the spectral radius of $A$, which is the maximum modulus of eigenvalue of $A$. 
$O$ is a matrix of $0$'s, $\be$ is a column vector of $1$'s and $\bzero$ is a column vector of $0$'s; their dimensions, which are finite or countably infinite, are determined in context. $I$ is the identity matrix.

%
%
\section{Nonnegative block tri-diagonal matrix and its properties} \label{sec:RandGmatrix}

Consider a nonnegative block tri-diagonal matrix $Q$ defined as 
\[
Q 
= \begin{pmatrix}
A_0 & A_1 & & & \cr
A_{-1} & A_0 & A_1 & & \cr
& A_{-1} & A_0 & A_1 & \cr
& & \ddots & \ddots & \ddots 
\end{pmatrix}, 
\]
where $A_{-1}$, $A_0$ and $A_1$ are nonnegative square matrices with a countable dimension, i.e., for $k\in\{-1,0,1\}$, $A_k=(a_{k,i,j},i,j\in\mathbb{Z}_+)$ and every $a_{k,i,j}$ is nonnegative. 
We define a matrix $A_*$ as 
\[
A_*=A_{-1}+A_0+A_1. 
\]
%
%
Hereafter, we adopt the policy to give a minimal assumption in each place. First, we give the following conditions. 
\begin{condition} \label{cond:As_finite} 
\begin{itemize}
\item[(a1)] Both $A_{-1}$ and $A_1$ are nonzero matrices.
\end{itemize}
\end{condition}

\begin{condition}
\begin{itemize}
\item[(a2)] All iterates of $A_*$ are finite, i.e., for any $n\in\mathbb{Z}_+$, $A_*^n<\infty$. 
\end{itemize}
\end{condition}

Condition (a1) makes $Q$ a \textit{true block tri-diagonal} matrix. 
Under condition (a2), all multiple products of $A_{-1}$, $A_0$ and $A_1$ becomes finite, i.e., for any $n\in\mathbb{N}$ and for any $\bi_{(n)}=(i_1,i_2, ..., i_n)\in\{-1,0,1\}^n$, $A_{i_1} A_{i_2} \cdots A_{i_n}<\infty$. 
Hence, for the triplet $\{ A_{-1}, A_0,  A_1\}$, we can define a matrix $R$ corresponding to the rate matrix of a QBD process and a matrix $G$ corresponding to the G-matrix. 
%
%
If $\cp(A_*)<\infty$, discussions for $Q$ may be reduced to probabilistic arguments. For example, if there exist an $s>0$ and positive vector $\bv$ such that $s A_* \bv\le \bv$, then $\Delta_{\bv}^{-1} A_* \Delta_{\bv}$ becomes stochastic or substochastic, where $\Delta_{\bv}=\diag\, \bv$, and discussion for the triplet $\{ A_{-1}, A_0,  A_1\}$ can be replaced with that for $\{ \Delta_{\bv}^{-1}A_{-1}\Delta_{\bv}, \Delta_{\bv}^{-1}A_0\Delta_{\bv},  \Delta_{\bv}^{-1}A_1\Delta_{\bv} \}$. However, in order to make discussion simple, we directly treat $\{ A_{-1}, A_0,  A_1 \}$ and do not use probabilistic arguments. 

Define the following sets of index sequences: for $n\ge 1$ and for $m\ge 1$, 
\begin{align*}
&\scrI_n = \biggl\{\bi_{(n)}\in\{-1,0,1\}^n;\ \sum_{l=1}^k i_l\ge 0\ \mbox{for $k\in\mathbb{N}_{n-1}$}\ \mbox{and} \sum_{l=1}^n i_l=0 \biggr\}, \\
&\scrI_{D,m,n} = \biggl\{\bi_{(n)}\in\{-1,0,1\}^n;\ \sum_{l=1}^k i_l\ge -m+1\ \mbox{for $k\in\mathbb{N}_{n-1}$}\ \mbox{and} \sum_{l=1}^n i_l=-m \biggr\}, \\
&\scrI_{U,m,n} = \biggl\{\bi_{(n)}\in\{-1,0,1\}^n;\ \sum_{l=1}^k i_l\ge 1\ \mbox{for $k\in\mathbb{N}_{n-1}$}\ \mbox{and} \sum_{l=1}^n i_l=m \biggr\}, 
\end{align*}
where $\bi_{(n)}=(i_1,i_2,...,i_n)$. Consider a QBD process $\{(X_n,J_n)\}$ on the state space $\mathbb{Z}_+^2$, where $X_n$ is the level and $J_n$ the phase. The set $\scrI_n $ corresponds to the set of all paths of the QBD process on which $X_0=l>0$, $X_k\ge l$ for $k\in\mathbb{N}_{n-1}$ and $X_n=l$, i.e., the level process visits state $l$ at time $n$ without entering states less than $l$ before time $n$. 
The set $\scrI_{D,m,n}$ corresponds to the set of all paths on which $X_0=l>m$, $X_k\ge l-m+1$ for $k\in\mathbb{N}_{n-1}$ and $X_n=l-m$, and $\scrI_{U,m,n}$ to that of all paths on which $X_0=l>0$, $X_k\ge l+1$ for $k\in\mathbb{N}_{n-1}$ and $X_n=l+m$. 
For $n\ge 1$, define $Q_{0,0}^{(n)}$, $D^{(n)}$ and $U^{(n)}$ as 
\begin{align*}
&Q_{0,0}^{(n)} = \sum_{\bi_{(n)}\in\scrI_n} A_{i_1} A_{i_2} \cdots A_{i_n},\quad
D^{(n)} = \sum_{\bi_{(n)}\in\scrI_{D,1,n}} A_{i_1} A_{i_2} \cdots A_{i_n}, \\
&U^{(n)} = \sum_{\bi_{(n)}\in\scrI_{U,1,n}} A_{i_1} A_{i_2} \cdots A_{i_n}. 
\end{align*}
Under (a2), $Q_{0,0}^{(n)}$, $D^{(n)}$ and $U^{(n)}$ are finite for every $n\ge 1$. Define $N$, $R$ and $G$ as 
\begin{align*}
&N = \sum_{n=0}^\infty Q_{0,0}^{(n)},\quad 
G = \sum_{n=1}^\infty D^{(n)}, \quad 
R = \sum_{n=1}^\infty U^{(n)},
\end{align*}
where $Q_{0,0}^{(0)}=I$. We call $N$, $G$ and $R$ the N-matrix, G-matrix and R-matrix generated from the triplet $\{ A_{-1}, A_0,  A_1\}$, respectively. 
The following properties hold. 
\begin{lemma} \label{le:RandGmatrix_equations}
Assume (a1) and (a2). Then, $N$, $G$ and $R$ satisfy the following equations, including the case where both the sides of the equations diverge. 
\begin{align}
&R =  A_1 N, \\
&G = N  A_{-1}, \\
&R = R^2  A_{-1}+R  A_0+ A_1, \label{eq:Rmatrix_equation0} \\
&G =  A_{-1}+ A_0 G+ A_1 G^2, \label{eq:Gmatrix_equation0} \\
&N = I+A_0 N+A_1 G N = I+N A_0+N A_1 G. \label{eq:NandH_relation}
\end{align}
\end{lemma}

To make this paper self-contained, we give a proof of the lemma in Appendix \ref{sec:proof_RandGmatrix_equations}. 
From (\ref{eq:NandH_relation}), $N\ge I$ and $N\ge A_0 N+A_1 G N \ge A_0 +A_1 G$. Hence, if $N$ is finite, then $A_0 +A_1 G$ is also finite and we obtain 
\begin{equation}
(I-A_0-A_1 G)N=N(I-A_0-A_1 G)=I
\label{eq:NandH_relation2}
\end{equation}
We will use equation (\ref{eq:NandH_relation}) in this form. Here, we should note that much attention must be paid to matrix manipulation since the dimension of matrices are countably infinite, e.g., see Appendix A of \cite{Takahashi01}. 

For $\theta\in\mathbb{R}$, define a matrix function $A_*(\theta)$ as 
\[
A_*(\theta) = e^{-\theta} A_{-1} + A_0 + e^\theta A_1, 
\]
where $A_*=A_*(0)$. This $A_*(\theta)$ corresponds to a Feynman-Kac operator if the triplet $\{A_{-1},A_0,A_1\}$ is a Markov additive kernel (see, e.g., \cite{Ney87}). 
For $R$ and $G$, we have the following identity corresponding to the RG decomposition for a Markov additive process, which is also called a Winer-Hopf factorization, see identity (5.5) of \cite{Miyazawa11} and references therein. 
\begin{lemma} \label{le:As_WHfac}
Assume (a1) and (a2). If $R$, $G$ and $N$ are finite, we have, for $\theta\in\mathbb{R}$, 
\begin{align}
I- A_*(\theta) &= (I-e^\theta R) (I-H) (I-e^{-\theta} G), 
\label{eq:As_WHfact_RG}
\end{align}
where $H = A_0+ A_1 G = A_0+ A_1 N  A_{-1}$. 
\end{lemma}

\begin{proof}
Using identities (\ref{eq:NandH_relation2}), we obtain
\begin{align}
I- A_*(\theta) 
&= I-e^{-\theta} (I-H)N A_{-1} - A_0-e^\theta A_1 \cr
&= (I-A_0-e^\theta A_1 - A_1 N A_{-1})(I-e^{-\theta} G) \cr
&= (I-A_0-e^\theta A_1 N(I-H) - A_1 N A_{-1})(I-e^{-\theta} G) \cr
&= \bigl( (I-e^\theta R) (I-H) \bigr) (I-e^{-\theta} G),
\label{eq:As_WHfact_RG2}
\end{align}
where we use the fact that, by Lemma \ref{le:RandGmatrix_equations}, every term appeared in the expression such as $H N A_{-1} = A_0 G+A_1 G^2$ and $A_1 N A_{-1}=A_1 G$ are finite. Analogously, we have 
\begin{align}
I- A_*(\theta) 
&= I-e^{-\theta} A_{-1} - A_0-e^\theta A_1 N(I-H) \cr
&= (I-e^\theta R)(I-A_0-e^{-\theta} A_{-1}-A_1 N A_{-1}) \cr
&= (I-e^\theta R)(I-A_0-e^{-\theta} (I-H)N A_{-1}-A_1 N A_{-1})) \cr
&= (I-e^\theta R) \bigl( (I-H) (I-e^{-\theta} G) \bigr).
\label{eq:As_WHfact_RG3}
\end{align}
As a result, we obtain (\ref{eq:As_WHfact_RG}).
\end{proof}

Consider the following matrix quadratic equations of $X$:
\begin{align}
& X = X^2 A_{-1}+X A_0+A_1, \label{eq:Rmatrix_equation1} \\
& X = A_{-1}+A_0 X+A_1 X^2. \label{eq:Gmatrix_equation1}
\end{align}
By Lemma \ref{le:RandGmatrix_equations}, $R$ and $G$ are solutions to equations (\ref{eq:Rmatrix_equation1}) and (\ref{eq:Gmatrix_equation1}), respectively. 
Consider the following sequences of matrices:
\begin{align}
X_0^{(1)}=O,\quad X_n^{(1)} = \bigl(X_{n-1}^{(1)}\bigr)^2 A_{-1}+X_{n-1}^{(1)} A_0+A_1,\ n\ge 1, 
\label{eq:Xn1_iteration} \\
X_0^{(2)}=O,\quad X_n^{(2)} = A_{-1}+A_0 X_{n-1}^{(2)}+A_1 \bigl(X_{n-1}^{(2)}\bigr)^2,\ n\ge 1. 
\label{eq:Xn2_iteration}
\end{align} 
Like the case of usual QBD process, we can demonstrate that both the sequences $\{X_n^{(1)}\}_{n\ge 0}$ and $\{X_n^{(2)}\}_{n\ge 0}$ are nondecreasing and that if a nonnegative solution $X^*$ to equation (\ref{eq:Rmatrix_equation1}) (resp.\ equation (\ref{eq:Gmatrix_equation1})) exists, then for any $n\ge 0$, $X^*\ge X_n^{(1)}$ (resp.\ $X^*\ge X_n^{(2)}$). 
Furthermore, letting $R_n$ and $G_n$ be defined as 
\[
R_n = \sum_{k=1}^n U^{(k)},\quad 
G_n = \sum_{k=1}^n D^{(k)},
\]
we can also demonstrate that, for any $n\ge 1$, $R_n\le X_n^{(1)}$ and $G_n\le X_n^{(2)}$ hold. Hence, we immediately obtain the following facts. 
%
%
%
\begin{lemma} \label{le:RandGmatrix_solutions}
Assume (a1) and (a2). Then, $R$ and $G$ are the minimum nonnegative solutions to equations (\ref{eq:Rmatrix_equation1}) and (\ref{eq:Gmatrix_equation1}), respectively. 
Furthermore, we have $R=\lim_{n\to\infty} X_n^{(1)}$ and $G=\lim_{n\to\infty} X_n^{(2)}$. 
\end{lemma}

%
If $A_*$ is irreducible, $A_*(\theta)$ is also irreducible for any $\theta\in\mathbb{R}$. We, therefore, give the following condition. 
\begin{condition}
\begin{itemize}
\item[(a3)] $A_*$ is irreducible.
\end{itemize}
\end{condition}
Let $\chi(\theta)$ be the reciprocal of the convergence parameter of $A_*(\theta)$, i.e., $\chi(\theta)=\cp(A_*(\theta))^{-1}$. 
We say that a positive function $f(x)$ is log-convex in $x$ if $\log f(x)$ is convex in $x$. A log-convex function is also a convex function. Since every element of $A_*(\theta)$ is log-convex in $\theta$, we see, by Lemma \ref{le:cp_convex} in Appendix \ref{sec:cp_convex}, that $\chi(\theta)$ satisfies the following property. 
\begin{lemma} \label{le:chi_convex}
Under (a1) through (a3), $\chi(\theta)$ is log-convex in $\theta\in\mathbb{R}$. 
\end{lemma}

Let $\gamma^\dagger$ be the infimum of $\chi(\theta)$, i.e., 
\[
\gamma^\dagger = \inf_{\theta\in\mathbb{R}} \chi(\theta),  
\]
and define a set $\bar{\Gamma}$ as 
\[
\bar{\Gamma} = \{\theta\in\mathbb{R}; \chi(\theta)\le 1\}. 
\]
By Lemma \ref{le:chi_convex}, if $\gamma^\dagger<1$ and $\bar{\Gamma}$ is bounded, then $\bar{\Gamma}$ is a line segment and there exist just two real solutions to equation $\chi(\theta)=\cp(A_*(\theta))^{-1}=1$. We denote the solutions by $\underline{\theta}$ and $\bar{\theta}$, where $\underline{\theta}<\bar{\theta}$. 
When $\gamma^\dagger=1$, we define $\underline{\theta}$ and $\bar{\theta}$ as $\underline{\theta}=\min\{\theta\in\mathbb{Z}; \chi(\theta)=1\}$ and $\bar{\theta}=\max\{\theta\in\mathbb{Z}; \chi(\theta)=1\}$, respectively. It is expected that $\underline{\theta}=\bar{\theta}$ if $\gamma^\dagger=1$, but it is not obvious. 
If $\gamma^\dagger\le1$ and $\bar{\Gamma}$ is bounded, there exists a $\theta\in\bar{\Gamma}$ such that $\gamma^\dagger=\chi(\theta)$. 
We give the following condition. 
\begin{condition}
\begin{itemize}
\item[(a4)] $\bar{\Gamma}$ is bounded. 
\end{itemize}
\end{condition}

If $A_{-1}$ (resp.\ $A_1$) is a zero matrix, every element of $A_*(\theta)$ is monotone increasing (resp.\ decreasing) in $\theta$ and $\bar{\Gamma}$ is unbounded. Hence, if $\gamma^\dagger\le 1$, condition (a4) implies (a1).
The following properties correspond to those in Lemma 2.3 of \cite{Kijima93}. 
\begin{proposition} \label{pr:Gmatrix_existence} 
Assume (a2) through (a4). 
\begin{itemize} 
\item[(i)] If $\gamma^\dagger\le 1$, then $R$ and $G$ are finite. 
\item[(ii)] If $R$ is finite and there exist a $\theta_0\in\mathbb{R}$ and nonnegative nonzero vector $\bu$ such that $e^{\theta_0} \bu^\top R=\bu^\top$, then $\gamma^\dagger\le 1$. 
\item[(ii')] If $G$ is finite and there exist a $\theta_0\in\mathbb{R}$ and nonnegative nonzero vector $\bv$ such that $e^{\theta_0}  G \bv=\bv$, then $\gamma^\dagger\le 1$. 
\end{itemize}
\end{proposition}
\begin{proof}
{\it Statement (i).}\quad Assume $\gamma^\dagger\le 1$ and let $\theta^\dagger$ be a real number satisfying $\chi(\theta^\dagger)=\gamma^\dagger$. Since $A_*(\theta^\dagger)$ is irreducible, by Lemma 1 and Theorem 1 of \cite{Pruitt64}, there exists a positive vector $\bu$ satisfying $(\gamma^\dagger)^{-1} \bu^\top A_*(\theta^\dagger)  \le \bu^\top$. 
For this $\bu$, we obtain, by induction using (\ref{eq:Xn1_iteration}), inequality $e^{\theta^\dagger} \bu^\top X_n^{(1)} \le  \bu^\top$ for any $n\ge 0$. 
Hence, the sequence $\{X_n^{(1)}\}$ is element-wise nondecreasing and bounded, and the limit of the sequence, which is the minimum nonnegative solution to equation (\ref{eq:Rmatrix_equation1}), exists. Existence of the minimum nonnegative solution to equation (\ref{eq:Gmatrix_equation1}) is analogously proved. As a result, by Lemma \ref{le:RandGmatrix_solutions}, both $R$ and $G$ are finite.

{\it Statements (ii) and (ii').}\quad 
Assume the condition of Statement (ii). Then, we have 
\begin{align}
\bu^\top 
&=  e^{\theta_0} \bu^\top R
=  e^{\theta_0} \bu^\top ( R^2 A_{-1}+R A_0+A_1 ) 
= \bu^\top A_*(\theta_0), 
\label{eq:z0uR}
\end{align}
and this leads us to $\gamma^\dagger\le \chi(\theta_0)=\cp(A_*(\theta_0))^{-1}\le 1$. 
Statement (ii') can analogously be proved. 
\end{proof}

\begin{remark}
In statement (ii) of Proposition \ref{pr:Gmatrix_existence}, if such $\theta_0$ and $\bu$ exist, then, by (\ref{eq:z0uR}) and irreducibility of $A_*(\theta_0)$, we have $\theta_0=\theta^\dagger$ and $\bu$ is positive. An analogous result also holds for statement (ii'). 
\end{remark}

\begin{remark} \label{rm:chi_unbounded}
Consider the following nonnegative matrix $P$:
\[
P 
= \begin{pmatrix}
 \ddots & \ddots & \ddots & & & \cr
& A_{-1} & A_0 & A_1 & & \cr
& & A_{-1} & A_0 & A_1 & \cr
& & & \ddots & \ddots & \ddots  
\end{pmatrix}.  
\]
If the triplet $\{A_{-1},A_0,A_1\}$ is a Markov additive kernel, this $P$ corresponds to the transition probability matrix of a Markov additive process governed by the triplet. By Proposition \ref{pr:chi_unbounded} in Appendix \ref{sec:chi_unbounded}, if $P$ is irreducible, then $\chi(\theta)$ is unbounded in both the directions of $\theta$ and $\bar{\Gamma}$ is bounded. 
\end{remark}

%
For the convergence parameters of $R$ and $G$, the following properties holds (for the case where the dimension of $A_*$ is finite, see Lemma 2.2 of \cite{He09} and Lemma 2.3 of  \cite{Miyazawa09}). 
\begin{lemma} \label{le:RandG_cp}
Assume (a2) through (a4).  If $\gamma^\dagger\le 1$ and $N$ is finite, then we have 
\begin{equation}
\cp(R) = e^{\bar{\theta}},\quad 
\cp(G) = e^{-\underline{\theta}}. 
\end{equation}
\end{lemma}
\begin{proof}
Since $\gamma^\dagger\le 1$ and $\bar{\Gamma}$ is bounded, $\bar{\theta}$ and $\underline{\theta}$ exist and they are finite. Furthermore, $R$ and $G$ are finite. 
For a $\theta\in\mathbb{R}$ such that $\chi(\theta)\le 1$, let $\bu$ be a positive vector satisfying $\bu^\top A_*(\theta)\le \bu^\top$. Such a $\bu$ exists since $A_*(\theta)$ is irreducible. As mentioned in the proof of Proposition \ref{pr:Gmatrix_existence}, for $X^{(1)}_n$ defined by (\ref{eq:Xn1_iteration}), if $\chi(\theta)\le 1$, then we have $e^\theta \bu^\top X^{(1)}_n\le \bu^\top$ for any $n\ge 0$ and this implies $e^\theta \bu^\top R\le \bu^\top$. 
Analogously, if $\chi(\theta)\le 1$, then there exists a positive vector $\bv$ satisfying $A_*(\theta)\bv\le \bv$ and we have $e^{-\theta} G \bv\le \bv$.
Therefore, setting $\theta$ at $\bar{\theta}$, we obtain $e^{\bar{\theta}} \bu^\top R\le \bu^\top$, and setting $\theta$ at $\underline{\theta}$, we obtain $e^{-\underline{\theta}} G \bv\le \bv$. Since $\bu$ and $\bv$ are positive, this leads us to $\cp(R)\ge e^{\bar{\theta}}$ and $\cp(G)\ge e^{-\underline{\theta}}$. 

Next, in order to prove $\cp(R)\le e^{\bar{\theta}}$, we apply a technique similar to that used in the proof of Theorem 1 of \cite{Pruitt64}. Suppose $\cp(R)>e^{\bar{\theta}}$. Then, there exists an $\varepsilon>0$ such that 
\[
\tilde{R}(\bar{\theta}+\varepsilon)=\sum_{n=0}^\infty e^{(\bar{\theta}+\varepsilon) n} R^n < \infty. 
\]
This $\tilde{R}(\bar{\theta}+\varepsilon)$ satisfies $e^{\bar{\theta}+\varepsilon} R \tilde{R}(\bar{\theta}+\varepsilon) = \tilde{R}(\bar{\theta}+\varepsilon)-I \le \tilde{R}(\bar{\theta}+\varepsilon)$. Hence, for $j\in\mathbb{Z}_+$, letting $\bv_j$ is the $j$-th column vector of $\tilde{R}(\bar{\theta}+\varepsilon)$, we have $e^{\bar{\theta}+\varepsilon} R \bv_j \le \bv_j$.
Furthermore, we have $e^{\bar{\theta}+\varepsilon} R \tilde{R}(\bar{\theta}+\varepsilon) \ge  e^{\bar{\theta}+\varepsilon} R \ge e^{\bar{\theta}+\varepsilon} X_1^{(1)} = e^{\bar{\theta}+\varepsilon} A_1$, and condition (a4) implies $A_1$ is nonzero. Hence, for some $j\in\mathbb{Z}_+$, both $R\bv_j$ and  $\bv_j$ are nonzero. Set $\bv$ at such a vector $\bv_j$. 
We have $\cp(G)\ge e^{-\underline{\theta}}> e^{-(\bar{\theta}+\varepsilon)}$. Hence, using (\ref{eq:NandH_relation2}) and (\ref{eq:As_WHfact_RG}), we obtain  
\begin{equation}
(I-A_*(\bar{\theta}+\varepsilon)) (I-e^{-(\bar{\theta}+\varepsilon)} G)^{-1} N \bv = (I-e^{\bar{\theta}+\varepsilon} R)\bv \ge \bzero, 
\end{equation}
where $\by=(I-e^{-(\bar{\theta}+\varepsilon)} G)^{-1} N \bv = \sum_{n=0}^\infty e^{-n (\bar{\theta}+\varepsilon)} G^n\,N \bv\ge N \bv$. 
Suppose $N \bv=\bzero$, then we have $R \bv = A_1 N \bv = \bzero$ and this contradicts that $R \bv$ is nonzero. Hence, $N \bv$ is nonzero and $\by$ is also nonzero and nonnegative. 
Since $A_*(\bar{\theta}+\varepsilon)$ is irreducible, the inequality $A_*(\bar{\theta}+\varepsilon) \by \le \by$ implies that $\by$ is positive and $\cp(A_*(\bar{\theta}+\varepsilon))\ge 1$. This contradicts that $\cp(A_*(\bar{\theta}+\varepsilon))=\chi(\bar{\theta}+\varepsilon)^{-1}<\chi(\bar{\theta})^{-1}=1$, and we obtain $\cp(R)\le e^{\bar{\theta}}$. 
In a similar manner, we can also obtain $\cp(G)\le e^{-\underline{\theta}}$, and this completes the proof.
\end{proof}

\begin{remark} \label{re:RandG_cp}
In Lemma \ref{le:RandG_cp}, it is not necessary that $R$ or $G$ is irreducible.
\end{remark}

Lemma \ref{le:RandG_cp} requires that $N$ is finite, but it cannot easily be verified since finiteness of $R$ and $G$ does not always imply that of $N$. We, therefore, introduce the following condition. 
\begin{condition}
\begin{itemize}
\item[(a5)] The nonnegative matrix $Q$ is irreducible. 
\end{itemize}
\end{condition}

Condition (a5) implies (a1), (a3) and (a4), i.e., under conditions (a2) and (a5), $A_{-1}$ and $A_1$ are nonzero, $A_*$ is irreducible and $\bar{\Gamma}$ is bounded. 
Let $\tilde{Q}$ be the fundamental matrix of $ Q$, i.e., $\tilde{Q}=\sum_{n=0}^\infty Q^n$. For $n\ge 0$, $Q_{0,0}^{(n)}$ is the $(0,0)$-block of $Q^n$, and $N$ is that of $\tilde{Q}$. Hence, we see that all the elements of $N$ simultaneously converge or diverge, finiteness of $R$ or $G$ implies that of $N$ and if $N$ is finite, it is positive. 
Furthermore, under (a2) and (a5), since $R$ is given as $R=A_1 N$ and $N$ is positive, each row of $R$ is zero or positive and we obtain the following proposition, which asserts that $R$ behaves just like an irreducible matrix. 
\begin{proposition} \label{pr:R_u}
Assume (a2) and (a5). If $R$ is finite, then it always satisfies one of the following two statements.
\begin{itemize}
\item[(i)] There exists a positive vector $\bu$ such that $e^{\bar{\theta}} \bu^\top R = \bu^\top$. 
\item[(ii)] $\sum_{n=0}^\infty e^{\bar{\theta}n} R^n <\infty$. 
\end{itemize}
\end{proposition}

Since the proof of this proposition is elementary and lengthy,  we put it in Appendix \ref{sec:proof_R_u}. By applying the same technique as that used in the proof of Theorem 4.1 of \cite{Kobayashi10}, we also obtain the following properties.
\begin{corollary} \label{co:cpRlimit}
Assume (a2) and (a5). For $i,j\in\mathbb{Z}_+$, if every element in the $i$-th row of $A_1$ is zero, we have $[R^n]_{i,j}=0$ for all $n\ge 1$; otherwise, we have $[R^n]_{i,j}>0$ for all $n\ge 1$ and 
\begin{equation}
\lim_{n\to\infty} ([R^n]_{i,j})^{\frac{1}{n}}=e^{-\bar{\theta}}. 
\label{eq:cpRlimit}
\end{equation}

\end{corollary}

To make this paper self-contained, we give a proof of the corollary in Appendix  \ref{sec:proof_R_u}.
By Theorem 2 of \cite{Pruitt64}, if the number of nonzero elements of each row of $A_*$ is finite, there exists a positive vector $\bu$ satisfying $\bu^\top A_*(\bar{\theta})=\bu^\top$. Also, if the number of nonzero elements of each column of $A_*$ is finite, there exists a positive vector $\bv$ satisfying $A_*(\underline{\theta})\bv=\bv$. To use this property, we give the following condition.
\begin{condition}
\begin{itemize}
\item[(a6)] The number of positive elements of each row and column of $A_*$ is finite.
\end{itemize}
\end{condition}
It is obvious that (a6) implies (a2). Under (a6), we can refine Proposition \ref{pr:Gmatrix_existence}, as follows.

\begin{proposition} \label{pr:RandG_existence2}
Assume (a5) and (a6). Then, $\gamma^\dagger\le 1$ if and only if $R$ and $G$ are finite.
\end{proposition}
\begin{proof}
By Proposition \ref{pr:Gmatrix_existence}, if $\gamma^\dagger\le 1$, then both $R$ and $G$ are finite. We, therefore, prove the converse.
Assume that $R$ and $G$ are finite. Then, $N$ is also finite and, by Lemma \ref{le:RandG_cp}, we have $\cp(R)=e^{\bar{\theta}}$. 
First, consider case (i) of Proposition \ref{pr:R_u} and assume that there exists a positive vector $\bu$ such that $e^{\bar{\theta}} \bu R = \bu$. Then, by statement (ii) of Proposition \ref{pr:Gmatrix_existence}, we have $\gamma^\dagger \le 1$.  
Next, consider case (ii) of Proposition \ref{pr:R_u} and assume $\sum_{n=0}^\infty e^{n \bar{\theta}} R^n <\infty$. Then, we have $(I-e^{\underline{\theta}} R)^{-1}=\sum_{k=0}^\infty e^{k \underline{\theta}} R^k<\infty$ since $\underline{\theta}\le \bar{\theta}$. Hence, we obtain, from (\ref{eq:NandH_relation2}) and (\ref{eq:As_WHfact_RG}) ,  
\begin{align}
 N (I-e^{\underline{\theta}} R)^{-1} (I-A_*(\underline{\theta})) = (I-e^{-\underline{\theta}} G). \label{eq:AsGandR_relation2}
\end{align}
Under the assumption of the proposition, there exists a positive vector $\bv$ satisfying $A_*(\underline{\theta})\bv=\bv$ since $\cp(A_*(\underline{\theta}))=1$. Hence, from (\ref{eq:AsGandR_relation2}), we obtain, for this $\bv$, $e^{-\underline{\theta}} G\bv =\bv$, and by statement (ii') of Proposition \ref{pr:Gmatrix_existence}, we have $\gamma^\dagger \le 1$. 
This completes the proof. 
\end{proof}

Recall that $\gamma^\dagger$ is defined as $\gamma^\dagger=\inf_{\theta\in\mathbb{R}}\cp(A_*(\theta))^{-1}$. Since if $Q$ is irreducible, all the elements of $\tilde{Q}=\sum_{n=0}^\infty Q^n$ simultaneously converge or diverge, we obtain, from Proposition \ref{pr:RandG_existence2}, the following property. 
\begin{proposition} \label{pr:Q_cp}
Assume (a5) and (a6). Then, $\gamma^\dagger\le 1$ if and only if $\tilde{Q}$ is finite. 
\end{proposition}
\begin{proof}
Under the assumption of the proposition, if $\gamma^\dagger \le 1$, then, by Proposition \ref{pr:RandG_existence2}, $R$ and $G$ are finite. Since $Q$ is irreducible, this implies that $N$ is finite and $\tilde{Q}$ is also finite. 
On the other hand, if $\tilde{Q}$ is finite, then $N$ is finite and $R$ and $G$ are also finite since the number of positive elements of each row of $A_1$ and that of each column of $A_{-1}$ are finite. Hence, by Proposition \ref{pr:RandG_existence2}, $\gamma^\dagger\le 1$ and this completes the proof. 
\end{proof}

For the convergence parameter of $Q$, we obtain, by this proposition, the following result. 
\begin{lemma} \label{le:Q_cp}
Under (a5) and (a6), we have $\cp(Q) = (\gamma^\dagger)^{-1} = \sup_{\theta\in\mathbb{R}} \cp(A_*(\theta))$ and $Q$ is $(\gamma^\dagger)^{-1}$-transient. 
\end{lemma}
\begin{proof}
For $\beta>0$, $\beta Q$ is a nonnegative block tri-diagonal matrix, whose block matrices are given by $\beta A_{-1}$, $\beta A_0$ and $\beta A_1$. Hence, the assumption of this lemma also holds for $\beta Q$. 
Define $\gamma(\beta)$ as 
\begin{equation}
\gamma(\beta) 
= \inf_{\theta\in\mathbb{R}} \cp(\beta A_*(\theta))^{-1} 
= \beta \inf_{\theta\in\mathbb{R}} \cp(A_*(\theta))^{-1}
= \beta \gamma^\dagger.
\end{equation}
By Proposition \ref{pr:Q_cp}, if $\gamma(\beta) = \beta \gamma^\dagger\le 1$, then the fundamental matrix of $\beta Q$, $\widetilde{\beta Q}$, is finite and $\cp(\beta Q)=\beta^{-1} \cp(Q)\ge 1$. Hence, if $\beta\le (\gamma^\dagger)^{-1}$, then $\cp(Q)\ge \beta$. Setting $\beta$ at $(\gamma^\dagger)^{-1}$, we obtain $\cp(Q)\ge (\gamma^\dagger)^{-1}$. 
Next we prove $\cp(Q)\le (\gamma^\dagger)^{-1}$. Suppose $\cp(Q)> (\gamma^\dagger)^{-1}$, then there exists an $\varepsilon>0$ such that the fundamental matrix of $((\gamma^\dagger)^{-1}+\varepsilon) Q$ is finite. By Proposition \ref{pr:Q_cp}, this implies 
\begin{equation}
\gamma((\gamma^\dagger)^{-1}+\varepsilon)) = ((\gamma^\dagger)^{-1}+\varepsilon) \gamma^\dagger = 1+\varepsilon \gamma^\dagger \le 1, 
\end{equation}
and we obtain $\gamma^\dagger\le 0$. This contradicts $\gamma^\dagger> 0$, which is obtained from the irreducibility of $A_*$. Hence, we obtain $\cp(Q)\le (\gamma^\dagger)^{-1}$.
Setting $\beta$ at $(\gamma^\dagger)^{-1}$, we have $\gamma(\beta)=\gamma((\gamma^\dagger)^{-1})\le 1$ and the fundamental matrix of $\beta Q=(\gamma^\dagger)^{-1}Q$ is finite. This means $Q$ is $(\gamma^\dagger)^{-1}$-transient.
\end{proof}

\begin{remark}
In the case where the phase space is finite, Lemma \ref{le:Q_cp} corresponds to Lemma 2.3 of \cite{Miyazawa15}. Assuming condition (a6), we extended that lemma to the case of infinite phase space. 
\end{remark}

%
%

%
\begin{remark}
For nonnegative block multi-diagonal matrices, a property similar to Lemma \ref{le:Q_cp} holds. We demonstrate it in the case of block quintuple-diagonal matrix.
Let $Q$ be a nonnegative block matrix defined as 
\[
Q 
= \begin{pmatrix}
A_0 & A_1 & A_2 & & & & \cr
A_{-1} & A_0 & A_1 & A_2 & & & \cr
A_{-2} & A_{-1} & A_0 & A_1 & A_2& & \cr
& A_{-2} & A_{-1} & A_0 & A_1 & A_2& \cr
& & \ddots & \ddots & \ddots & \ddots & \ddots 
\end{pmatrix}, 
\]
where $A_i,\,i\in\{-2,-1,0,1,2\}$, are nonnegative square matrices with a countable dimension.
For $\theta\in\mathbb{R}$, define a matrix function $A_*(\theta)$ as 
\begin{equation}
A_*(\theta) = \sum_{i=-2}^2 e^{i \theta} A_i.
\end{equation}
Then, assuming that $Q$ is irreducible and the number of positive elements of each row and column of $A_*(0)$ is finite, we can obtain 
\begin{equation}
\cp(Q) = \sup_{\theta\in\mathbb{R}} \cp(A_*(\theta)).
\label{eq:cpQ_5diagonal}
\end{equation}
Here we prove this equation. Define block matrices $\hat{A}_i,\,i\in\{-1,0,1\}$, as
\[
\hat{A}_{-1} = 
\begin{pmatrix} A_{-2} & A_{-1} \cr O & A_{-2} \end{pmatrix}, \quad
\hat{A}_0 = 
\begin{pmatrix} A_0 & A_1 \cr A_{-1} & A_0 \end{pmatrix}, \quad
\hat{A}_1 = 
\begin{pmatrix} A_2 & O \cr A_1 & A_2 \end{pmatrix},  
\]
then $Q$ is represented in block tri-diagonal form in terms of these block matrices. For $\theta\in\mathbb{R}$, define a matrix function $\hat{A}_*(\theta)$ as 
\[
\hat{A}_*(\theta) 
= e^{-\theta} \hat{A}_{-1} + \hat{A}_0 + e^{\theta} \hat{A}_1
= \begin{pmatrix} 
e^{-\theta} A_{-2} + A_0 + e^{\theta} A_2 & e^{-\theta/2} (e^{-\theta/2} A_{-1} + e^{\theta/2} A_1) \cr 
e^{\theta/2} (e^{-\theta/2} A_{-1} + e^{\theta/2} A_1) & e^{-\theta} A_{-2} + A_0 + e^{\theta} A_2
\end{pmatrix}, 
\]
then, by Lemma \ref{le:Q_cp}, we obtain $\cp(Q) = \sup_{\theta\in\mathbb{R}} \cp(\hat{A}_*(\theta))$. Hence, in order to prove equation (\ref{eq:cpQ_5diagonal}), it suffices to show that, for any $\theta\in\mathbb{R}$, 
\begin{align}
\cp(A_*(\theta/2)) 
&= \sup\{ \alpha\in\mathbb{R}_+; \alpha \bx A_*(\theta/2)\le \bx\ \mbox{for some}\ \bx>\bzero^\top \} \cr
&= \sup\{ \alpha\in\mathbb{R}_+; \alpha \hat{\bx} \hat{A}_*(\theta)\le \hat{\bx}\ \mbox{for some}\ \hat{\bx}>\bzero^\top \} 
= \cp(\hat{A}_*(\theta)). 
\label{eq:AshatAs}
\end{align}
For $\theta\in\mathbb{R}$ and $\alpha\in\mathbb{R}_+$, if $\alpha \bx A_*(\theta/2)\le \bx$ for some $\bx>\bzero^\top$, then, letting $\hat{\bx}=(\bx, e^{-\theta/2} \bx)$, we have $\alpha \hat{\bx} \hat{A}_*(\theta) \le \hat{\bx}$. 
On the other hand, if $\alpha \hat{\bx} \hat{A}_*(\theta) \le \hat{\bx}$ for some $\hat{\bx}=(\hat{\bx}_1,\hat{\bx}_2)>\bzero^\top$, then letting $\bx=\hat{\bx}_1+e^{\theta/2} \hat{\bx}_2$, we have $\alpha \bx A_*(\theta/2)\le \bx$. 
As a result, we obtain equations (\ref{eq:AshatAs}). 

\end{remark}

%
%
\section{Multi-dimensional QBD process} \label{sec:modelandresults}

Before explaining models, we introduce some notations. For a finite set $\alpha$, we denote by $\scrP(\alpha)$ the set of all subsets of $\alpha$, including the empty set. For $\beta\in\scrP(\alpha)$, we define $\beta^C$ as $\beta^C=\alpha\setminus\beta$. 
Let $d$ be the dimension of the QBD process and define a set $D$ as $D=\{1,2,...,d\}$. We use $\scrP(D)$ as an index set. 
For a vector $\bx\in\mathbb{R}^d$, we denote by $x(l)$ the $l$-th element of $\bx$ and, for $\alpha\in\scrP(D)$, denote by $\bx(\alpha)$ a part of $\bx$ specified by $\alpha$, i.e., $\bx(\alpha)=(x(l),l\in \alpha)$. For example, when $d=5$ and $\alpha=\{2,4,5\}$, we have $\bx(\alpha)=(x(2),x(4),x(5))$, $\bx(\alpha^C)=(x(1),x(3))$ and $\bx=(\bx(\alpha),\bx(\alpha^C))$. Note that, if $\alpha=\emptyset$, $\bx(\alpha)$ means nothing.
For a set $\alpha$, we denote by $|\alpha|$ the cardinality of $\alpha$. We have $|\scrP(D)|=2^d$. 
For a real vector $\bx$ and a real number $c$, if every element of $\bx$ is equal to $c$, we express it as $\bx=c$. We analogously define $\bx<c$, $\bx>c$ and so on. 

\subsection{Model description} \label{sec:modeldescription}
%
Let $\{\bY_n\}=\{(\bX_n,J_n)\}$ be a $d$-dimensional QBD process on the state space $\mathbb{S}_+=\mathbb{Z}_+^d\times S_0$, where $S_0=\{1,2,...,s_0\}$ is the phase space, $\bX_n=(X_n(1),X_n(2),...,X_n(d))$ is the level state and $J_n$ is the phase state. We assume the cardinality of $S_0$, $s_0$, is finite. 
The level process $\{\bX_n\}$ is skip free in all directions, which means that, for $n\ge 0$, $\bX_{n+1}-\bX_n\in\{-1,0,1\}^d$. The $d$-dimensional QBD process is a Markov chain in which the transition probabilities of the level process vary according to the phase state. 
Divide $\mathbb{Z}_+^d$ into $2^d$ exclusive subsets defined by 
\[
\mathbb{B}^\alpha=\{\bx\in\mathbb{Z}_+^d; \bx(\alpha)>0,\ \bx(\alpha^C)=0 \},\ \alpha\in\scrP(D). 
\]
Since we have $\mathbb{B}^\alpha\cap\mathbb{B}^\beta=\emptyset$ for $\alpha\ne\beta$ and $\mathbb{Z}_+^d=\bigcup_{\alpha\in\scrP(D)} \mathbb{B}^\alpha$, the class $\{\mathbb{B}^\alpha; \alpha\in\scrP(D)\}$ is a partition of $\mathbb{Z}_+^d$. $\mathbb{B}^\emptyset$ is the set containing only the origin and $\mathbb{B}^D$ is the set of all positive points in $\mathbb{Z}_+^d$. 
Since the level process is skip free, the transition probabilities of $\{\bY_n\}$ are given by, for every $\bx\in\mathbb{Z}_+^d$ and $\alpha\in\scrP(D)$ such that $\bx\in\mathbb{B}^\alpha$ and for $i,j\in S_0$,
\begin{align}
&\mathbb{P}(\bY_{n+1}=(\bx+\bl,j)\,|\,\bY_n=(\bx,i)) \nonumber \\
&= \left\{ \begin{array}{ll} 
a_{\bl}^{\alpha,\beta}(i,j), & \mbox{if $\bl(\alpha)\in\{-1,0,1\}^{|\alpha|}$, $\bl(\alpha^C)\in\{0,1\}^{d-|\alpha|}$, $\bx+\bl\in\mathbb{B}^\beta$ for some $\beta\in\scrP(D)$}, \cr
0, & \mbox{otherwise}, 
\end{array} \right.
\label{eq:tp_calL}
\end{align}
where each $a_{\bl}^{\alpha,\beta}(i,j)$ is a nonnegative real number less than or equal to $1$. The transition probabilities $\{a_{\bl}^{\alpha,\beta}(i,j)\}$ satisfy several equalities so that the transition probability matrix $P$ defined below becomes stochastic, for example, for $\alpha\in\scrP(D)$ and $i\in S_0$, 
\begin{equation}
\sum_{\beta\in\scrP(\alpha^C)}\ \sum_{\bl(\alpha)\in\{-1,0,1\}}\ \sum_{\bl(\alpha^C)\in\{0,1\}}\ \sum_{k\in S_0} a_{\bl}^{\alpha,\alpha\cup\beta}(i,k) =1.
\label{eq:sum1_a}
\end{equation}
We denote by $A_{\bl}^{\alpha,\beta}$ the matrix whose $(i,j)$-element is $a_{\bl}^{\alpha,\beta}(i,j)$, i.e., $A_{\bl}^{\alpha,\beta}=(a_{\bl}^{\alpha\beta}(i,j); i,j\in S_0)$. 
We symbolically denote by $P$ the transition probability matrix of $\{\bY_n\}$, i.e., 
\[
P=\left( P_{\bx,\bx'}; \bx,\bx'\in\mathbb{Z}_+^d \right), 
\]
where $P_{\bx,\bx'}=(p_{(\bx,j),(\bx',j')}; j,j'\in S_0)$ and $p_{(\bx,j),(\bx',j')}=\mathbb{P}(\bY_1=(\bx,j)\,|\,\bY_0=(\bx',j'))$. For every $\bx\in\mathbb{Z}_+^d$ and $\alpha\in\scrP(D)$ such that $\bx\in\mathbb{B}^\alpha$,
\[
P_{\bx,\bx+\bl} 
= \left\{ \begin{array}{ll} 
A^{\alpha,\beta}_{\bl}, & \mbox{ if $\bl(\alpha)\in\{-1,0,1\}^{|\alpha|}$, $\bl(\alpha^C)\in\{0,1\}^{d-|\alpha|}$, $\bx+\bl\in\mathbb{B}^\beta$ for some $\beta\in\scrP(D)$}, \cr
O, & \mbox{otherwise}.
\end{array} \right.
\]
This means that the matrix $P$ has \textit{a multiple block tri-diagonal structure} and it is clue to analyzing the $d$-dimensional QBD process. Since $S_0$ is finite, the number of positive elements of each row and column of $P$ is finite. Hence, $P$ and matrices come from $P$ satisfy condition (a6) in Section \ref{sec:RandGmatrix}. 
%

We assume the following condition throughout this section. 
\begin{assumption} \label{as:QBD_stable}
The $d$-dimensional QBD process $\{\bY_n\}$ is irreducible and positive recurrent. 
\end{assumption}
We denote by $\bnu$ the stationary distribution of $\{\bY_n\}$, where $\bnu=(\bnu_{\bx},\bx\in\mathbb{Z}_+^d)$, $\bnu_{\bx}=(\nu_{\bx,j}, j\in S_0)$ and $\nu_{\bx,j}$ is the stationary probability that the $d$-dimensional QBD process is in the state of $(\bx,j)$ in steady state. 
Our main aim in this section is to give lower bounds for the directional asymptotic decay rates of the stationary distribution, i.e., for $k\in D$,  to give a positive number $\theta_k$ satisfies, for $\bx(\{k\}^C)\in\mathbb{Z}_+^{d-1}$ and $j\in S_0$, 
\begin{equation}
\liminf_{n\to\infty} \frac{1}{n} \log \nu_{(x(k)=n,\bx(\{k\}^C)),j} \ge -\theta_k.
\end{equation}

%
%
\subsection{Matrix geometric solutions for the stationary distribution} \label{sec:MGsolution}

For $k\in D$, we represent the $d$-dimensional QBD process $\{\bY_n\}=\{(\bX_n,J_n)\}$ as a one-dimensional QBD process 
\[
\{\bY^{\{k\}}_n\}=\{(X_n(k),(\bX_n(\{k\}^C),J_n))\},
\]
where $X_n(k)$ is the level state and $(\bX_n(\{k\}^C),J_n)=((X_n(l),l\in D\setminus\{k\}),j)$ is the phase state. Since the level process $\{X_n(k)\}$ is skip free, the transition probability matrix $P$ is represented, in block form, as
\begin{equation} \label{eq:P1_blockform}
P = 
\begin{pmatrix}
\hat{P}^{\{k\}}_0 & \hat{P}^{\{k\}}_1 & & & \cr
\hat{P}^{\{k\}}_{-1} & P^{\{k\}}_0 & P^{\{k\}}_1 & & \cr
& P^{\{k\}}_{-1} & P^{\{k\}}_0 & P^{\{k\}}_1 & \cr
& & \ddots & \ddots & \ddots 
\end{pmatrix}, 
\end{equation}
where, for $l\in\{0,1\}$ and $l'\in\{-1,0,1\}$, 
\begin{align*}
&\hat{P}^{\{k\}}_l = \left( P_{\bx,\bx'}; x(k)=0,x'(k)=l,\,\bx(\{k\}^C),\bx'(\{k\}^C)\in\mathbb{Z}_+^{d-1} \right), \\
&\hat{P}^{\{k\}}_{-1} = \left( P_{\bx,\bx'}; x(k)=1,x'(k)=0,\,\bx(\{k\}^C),\bx'(\{k\}^C)\in\mathbb{Z}_+^{d-1} \right), \\
&P^{\{k\}}_{l'} = \left( P_{\bx,\bx'}; x(k)=2,x'(k)=2+l',\,\bx(\{k\}^C),\bx'(\{k\}^C)\in\mathbb{Z}_+^{d-1} \right).
\end{align*}
We also represent the stationary distribution $\bnu$ in the same block form, i.e., 
\[
\bnu^{\{k\}} = (\bnu^{\{k\}}_n, n\in\mathbb{Z}_+),\quad 
\bnu^{\{k\}}_n = \left( \bnu_{\bx}; x(k)=n, \bx(\{k\}^C)\in\mathbb{Z}_+^{d-1} \right).
\]
Let $R^{\{k\}}$ be the rate matrix generated from the triplet $\{ P^{\{k\}}_{-1},P^{\{k\}}_0,P^{\{k\}}_1 \}$, then $\bnu^{\{k\}}$ is given in the following matrix geometric form:
\begin{equation}
\bnu^{\{k\}}_n = \bnu^{\{k\}}_1 (R^{\{k\}})^{n-1},\ n\ge 1.
\end{equation}
%
%
%
%
Define a truncation of $P$, $Q^{\{k\}}$, as
\begin{equation} \label{eq:Q1_blockform}
Q^{\{k\}} = 
\begin{pmatrix}
P^{\{k\}}_0 & P^{\{k\}}_1 & & \cr
P^{\{k\}}_{-1} & P^{\{k\}}_0 & P^{\{k\}}_1 & \cr
& \ddots & \ddots & \ddots 
\end{pmatrix},  
\end{equation}
and a matrix function $P^{\{k\}}_*(\theta)$ and a set $\bar{\Gamma}^{\{k\}}$ as 
\begin{align}
&P^{\{k\}}_*(\theta) = e^{-\theta} P^{\{k\}}_{-1} + P^{\{k\}}_0 + e^\theta P^{\{k\}}_1, \label{eq:Pk_def}\\
&\bar{\Gamma}^{\{k\}} = \left\{ \theta\in\mathbb{R}; \cp(P^{\{k\}}_*(\theta))^{-1} \le 1 \right\}. 
\end{align}

In order to obtain lower bounds for the directional asymptotic decay rates of the stationary distribution, we assume the following condition, which corresponds to condition (a5) in Section \ref{sec:RandGmatrix}.
\begin{assumption} \label{as:PkRk_irreducible}
For $k\in D$, $Q^{\{k\}}$ is irreducible. 
\end{assumption}

Let $k$ be an integer in $D$. Under this assumption, $Q^{\{k\}}$ satisfies conditions (a2) through (a4) and if $R^{\{k\}}$ is finite, the N-matrix generated from the triplet $\{ P^{\{k\}}_{-1},P^{\{k\}}_0,P^{\{k\}}_1 \}$ is also finite. Hence, by Lemma \ref{le:RandG_cp}, if $\bar{\Gamma}^{\{k\}}$ contains at least one positive number, we have 
\[
\cp(R^{\{k\}}) = e^{\bar{\theta}^{\{k\}}},\quad 
\bar{\theta}^{\{k\}} = \max\!\left\{ \theta\in\mathbb{R}; \cp(P^{\{k\}}_*(\theta))^{-1}\le 1 \right\}> 0. 
\]
Regard the triplet $\{ P^{\{k\}}_{-1},P^{\{k\}}_0,P^{\{k\}}_1 \}$ as a Markov additive kernel (MA-kernel for short) and let 
\[
\{\tilde{\bY}^{\{k\}}_n\}= \left\{ (\tilde{X}^{\{k\}}_n(k), (\tilde{\bX}^{\{k\}}_n(\{k\}^C),\tilde{J}_n)) \right\}
\]
be a Markov additive process (MA-process for short) on the state space $\mathbb{Z}\times(\mathbb{Z}_+^{d-1}\times S_0)$, governed by the Markov additive kernel. 
Then, the occupation measure for the Markov additive process is given by $\{(R^{\{k\}})^l, l\ge 0\}$, i.e., for a stopping time $\tau$ defined as $\tau=\inf\{n\ge 1; \tilde{X}^{\{k\}}_n(k)=0\}$ and for $(\bx(\{k\}^C),j),(\bx'(\{k\}^C),j')\in\mathbb{Z}_+^{d-1}\times S_0$, 
\begin{align}
&[(R^{\{k\}})^l]_{(\bx(\{k\}^C),j),(\bx'(\{k\}^C),j')} \cr
&\quad = \mathbb{E}\!\left[ \sum_{n=0}^{\tau-1} 1\Big(\tilde{\bY}^{\{k\}}_n=(l,(\bx'(\{k\}^C),j'))\Big) \,\Big|\, \tilde{\bY}^{\{k\}}_0=(0,(\bx(\{k\}^C),j)) \right], 
\end{align}
where $1(\cdot)$ is an indicator function. 
%
We have, by Corollary \ref{co:cpRlimit}, for $(\bx(\{k\}^C),j)\in \mathbb{Z}_+^{d-1}\times S_0$ and for $(\bx'(\{k\}^C),j')\in \mathbb{Z}_+^{d-1}\times S_0$ such that the $(\bx'(\{k\}^C),j')$-row of $P_1^{\{k\}}$ contains at least one positive element, 
\begin{equation}
\lim_{n\to\infty} \frac{1}{n} \log [(R^{\{k\}})^n]_{(\bx'(\{k\}^C),j'),(\bx(\{k\}^C),j)} = - \bar{\theta}^{\{k\}}. 
\label{eq:Rk_limit}
\end{equation}
From  the matrix geometric solution for the stationary distribution, we obtain, for $(\bx(\{k\}^C),j)\in \mathbb{Z}_+^{d-1}\times S_0$ and for $(\bx'(\{k\}^C),j')\in \mathbb{Z}_+^{d-1}\times S_0$ such that the $(\bx'(\{k\}^C),j')$-row of $P_1^{\{k\}}$ contains at least one positive element, 
\begin{equation}
\nu_{((x(k)=n,\bx(\{k\}^C)),j)} \ge \nu_{((x'(k)=1,\bx'(\{k\}^C)),j')} [(R^{\{k\}})^{n-1}]_{(\bx'(\{k\}^C),j'),(\bx(\{k\}^C),j)}, 
\label{eq:nu_inequality}
\end{equation}
where $\nu_{((x'(k)=1,\bx'(\{k\}^C)),j')}>0$ since $P$ is irreducible. Hence, from (\ref{eq:Rk_limit}) and (\ref{eq:nu_inequality}), we obtain
\begin{equation}
\liminf_{n\to\infty} \frac{1}{n} \log \nu_{((x(k)=n,\bx(\{k\}^C)),j)} \ge - \bar{\theta}^{\{k\}}. 
\label{eq:nu_limit1}
\end{equation}
In the following subsections, using the multiple block tri-diagonal structure of $P$, we will obtain upper bounds for $\bar{\theta}^{\{k\}}$.

%
%
\subsection{Lower bounds for the asymptotic decay rates}

Hereafter, we denote by $\btheta$ a $d$-dimensional real vector. For $k_1\in D$, consider the one-dimensional QBD process $\{\bY^{\{k_1\}}\}=\{(X_n(k_1),(\bX_n(\{k_1\}^C),J_n))\}$ and matrix function $P^{\{k_1\}}_*(\theta(k_1))$, defined in the previous subsection. From (\ref{eq:Pk_def}), we see that $P^{\{k_1\}}_*(\theta(k_1))$ is represented as
\begin{equation}
P^{\{k_1\}}_*(\theta(k_1)) = \left( P^{\{k_1\}}_{*,\bx(\{k_1\}^C),\bx'(\{k_1\}^C)}(\theta(k)); \bx(\{k_1\}^C),\bx'(\{k_1\}^C)\in\mathbb{Z}_+^{d-1} \right), 
\label{eq:Pks_theta}
\end{equation}
where $\theta(k_1)$ is the $k_1$-th element of $\btheta$ and
\[
P^{\{k_1\}}_{*,\bx(\{k_1\}),\bx'(\{k_1\}^C)}(\theta(k_1)) = \sum_{l\in\{-1,0,1\}} e^{l \theta(k_1)} P_{(x(k_1)=2,\bx(\{k_1\}^C)),(x'(k_1)=2+l,\bx'(\{k_1\}^C))}. 
\]
In general,  we define, for a nonempty set $\alpha\in\scrP(D)$, a matrix function $P^{\alpha}_*(\btheta(\alpha))$ as
\begin{equation}
P^{\alpha}_*(\btheta(\alpha)) = \left( P^{\alpha}_{*,\bx(\alpha^C),\bx'(\alpha^C)}(\btheta(\alpha)); \bx(\alpha^C),\bx'(\alpha^C)\in\mathbb{Z}_+^{d-|\alpha|} \right), 
\label{eq:Palpha_theta}
\end{equation}
where 
\[
P^{\alpha}_{*,\bx(\alpha^C),\bx'(\alpha^C)}(\btheta(\alpha)) 
= \sum_{\bl(\alpha)\in\{-1,0,1\}^{|\alpha|}} e^{\langle \bl(\alpha), \btheta(\alpha) \rangle} P_{(\bx(\alpha)=2,\bx(\alpha^C)),(\bx'(\alpha)=\bx(\alpha)+\bl(\alpha),\bx'(\alpha^C))}. 
\]
where $\langle \ba,\bb \rangle$ is the inner product of vectors $\ba$ and $\bb$. 
In terms of $P^{\alpha}_*(\btheta(\alpha))$, we also define a point set $\bar{\Gamma}^{\alpha}$ as
\begin{equation}
\bar{\Gamma}^{\alpha} = \left\{ \btheta\in\mathbb{R}^d; \cp(P^{\alpha}_*(\btheta(\alpha)))^{-1} \le 1 \right\}.
\label{eq:barGamma_alpha}
\end{equation}

For $\alpha\in\scrP(D)$ such that $1\le |\alpha|\le d-1$ and for $k\in D\setminus\alpha$, we consider a relation between $P^{\alpha}_*(\btheta(\alpha)) $ and $P^{\alpha\cup\{k\}}_*(\btheta(\alpha\cup\{k\}))$. 
Represent the subprocess $\{(\bX_n(\alpha^C),J_n)\}$  of the original $d$-dimensional QBD process as a one-dimensional QBD process $\{(X_n(k),(\bX_n(\alpha^C\setminus\{k\}),J_n))\}$, where $X_n(k)$ is the level state and $(\bX_n(\alpha^C\setminus\{k\}),J_n)$ the phase state.  Nonnegative matrix $P^{\alpha}_*(\bzero_{|\alpha|})$ is the transition probability matrix for the process $\{(\bX_n(\alpha^C),J_n)\}$ when $\bX_n(\alpha^C)>1$, where $\bzero_{|\alpha|}$ is a $|\alpha|$-dimensional vector of $0$'s.
Since the level process $\{X_n(k)\}$ is skip free, we see from (\ref{eq:Pks_theta}) that $P^{\alpha}_*(\btheta(\alpha))$ is represented, in block tri-diagonal form, as
\begin{equation} 
P^{\alpha}_*(\btheta(\alpha)) = 
\begin{pmatrix}
\hat{P}^{\alpha,\{k\}}_{*,0}(\btheta(\alpha)) & \hat{P}^{\alpha,\{k\}}_{*,1}(\btheta(\alpha)) & & & \cr
\hat{P}^{\alpha,\{k\}}_{*,-1}(\btheta(\alpha)) & P^{\alpha,\{k\}}_{*,0}(\btheta(\alpha)) & P^{\alpha,\{k\}}_{*,1}(\btheta(\alpha)) & & \cr
& P^{\alpha,\{k\}}_{*,-1}(\btheta(\alpha)) & P^{\alpha,\{k\}}_{*,0}(\btheta(\alpha)) & P^{\alpha,\{k\}}_{*,1}(\btheta(\alpha)) & \cr
& & \ddots & \ddots & \ddots 
\end{pmatrix}, 
\label{eq:Palpha_blockform}
\end{equation}
where, for $l\in\{0,1\}$ and $l'\in\{-1,0,1\}$,  
\begin{align*}
&\hat{P}^{\alpha,\{k\}}_{*,l}(\btheta(\alpha)) \cr
&\quad= \left( P^{\alpha}_{*,\bx(\alpha^C),\bx'(\alpha^C)}(\btheta(\alpha)); x(k)=0,x'(k)=l,\,\bx(\alpha^C\setminus\{k\}),\bx'(\alpha^C\setminus\{k\})\in\mathbb{Z}_+^{d-|\alpha|-1} \right), \\
&\hat{P}^{\alpha,\{k\}}_{*,-1}(\btheta(\alpha)) \cr
&\quad= \left( P^{\alpha}_{*,\bx(\alpha^C),\bx'(\alpha^C)}(\btheta(\alpha)); x(k)=1,x'(k)=0,\,\bx(\alpha^C\setminus\{k\}),\bx'(\alpha^C\setminus\{k\})\in\mathbb{Z}_+^{d-|\alpha|-1} \right), \\
&P^{\alpha,\{k\}}_{*,l}(\btheta(\alpha)) \cr
&\quad= \left( P^{\alpha}_{*,\bx(\alpha^C),\bx'(\alpha^C)}(\btheta(\alpha)); x(k)=2,x'(k)=2+l,\,\bx(\alpha^C\setminus\{k\}),\bx'(\alpha^C\setminus\{k\})\in\mathbb{Z}_+^{d-|\alpha|-1} \right).
\end{align*}
Then, $P^{\alpha\cup\{k\}}_*(\btheta(\alpha\cup\{k\}))$ is given as 
\begin{equation}
P^{\alpha\cup\{k\}}_*(\btheta(\alpha\cup\{k\})) = e^{-\theta(k)} P^{\alpha,\{k\}}_{*,-1}(\btheta(\alpha))+P^{\alpha,\{k\}}_{*,0}(\btheta(\alpha))+ e^{\theta(k)} P^{\alpha,\{k\}}_{*,1}(\btheta(\alpha)). 
\label{eq:Palphak_def}
\end{equation}
Hence, we see that $P^{\alpha\cup\{k\}}_*(\btheta(\alpha\cup\{k\}))$ is generated from $P^{\alpha}_*(\btheta(\alpha))$ through (\ref{eq:Palphak_def}). 
Define  a truncation of $P^{\alpha}_*(\btheta(\alpha))$, $Q^{\alpha,\{k\}}_*(\btheta(\alpha))$, as 
\begin{equation} \label{eq:Qalphak_blockform}
Q^{\alpha,\{k\}}_*(\btheta(\alpha)) = 
\begin{pmatrix}
P^{\alpha,\{k\}}_{*,0}(\btheta(\alpha)) & P^{\alpha,\{k\}}_{*,1}(\btheta(\alpha)) & & \cr
P^{\alpha,\{k\}}_{*,-1}(\btheta(\alpha)) & P^{\alpha,\{k\}}_{*,0}(\btheta(\alpha)) & P^{\alpha,\{k\}}_{*,1}(\btheta(\alpha)) & \cr
& \ddots & \ddots & \ddots 
\end{pmatrix}.  
\end{equation}
We assume the following condition.
\begin{assumption} \label{as:Qalphak_irreducible}
For every $\alpha\in\scrP(D)$ such that $1\le|\alpha|\le d-1$ and for every $k\in D\setminus\alpha$, $Q^{\alpha,\{k\}}_*(\bzero_{|\alpha|})$ is irreducible. 
\end{assumption}

In the following subsection, we will consider a more tractable condition instead of this assumption. Under this assumption, we obtain upper bounds for $\bar{\theta}^{\{k\}}$, as follows.
\begin{proposition} \label{pr:Rk_cp}
For $k\in D$ and for $\alpha\in\scrP(D\setminus\{k\})$, 
\begin{equation}
\log \cp(R^{\{k\}}) 
= \bar{\theta}^{\{k\}} 
\le \max\!\left\{ \theta(k); \btheta\in\bar{\Gamma}^{\{k\}\cup\alpha} \right\}.
\label{eq:cpRk_upper}
\end{equation}
\end{proposition}
\begin{proof}
Let $\alpha\in\scrP(D\setminus\{k\})$. Without loss of generality, we assume $\alpha=\{k_1,k_2,...,k_{n_0}\}$, where $n_0=|\alpha|$. Let $\alpha_0=\emptyset$ and, for $1\le n \le n_0$, let $\alpha_n=\{k_1,k_2,...,k_n\}$. We prove (\ref{eq:cpRk_upper}) by induction. 
First, we have
\[
\bar{\theta}^{\{k\}} 
= \max\!\left\{ \theta(k)\in\mathbb{R}; \cp(P^{\{k\}}_*(\theta(k)))^{-1}\le 1 \right\}
= \max\!\left\{ \theta(k)\in\mathbb{R}; \btheta\in\bar{\Gamma}^{\{k\}} \right\}. 
\]
Next, for $n\le n_0-1$, assume that
\begin{align*}
\bar{\theta}^{\{k\}} 
&\le \max\!\left\{ \theta(k)\in\mathbb{R}; \inf_{\btheta(\alpha_n)\in\mathbb{R}^n} \cp(P^{\{k\}\cup\alpha_n}_*(\btheta(\{k\}\cup\alpha_n)))^{-1}\le 1 \right\} \cr
&= \max\!\left\{ \theta(k)\in\mathbb{R}; \btheta\in\bar{\Gamma}^{\{k\}\cup\alpha_n} \right\}. 
\end{align*}
Under Assumption \ref{as:Qalphak_irreducible}, $Q^{\{k\}\cup\alpha_n,\{k_{n+1}\}}_*(\btheta(\{k\}\cup\alpha_n))$ is irreducible and, by Lemma \ref{le:Q_cp}, we obtain 
\begin{align*}
\cp(P^{\{k\}\cup\alpha_n}_*(\btheta(\{k\}\cup\alpha_n)))
&\le \cp(Q^{\{k\}\cup\alpha_n,\{k_{n+1}\}}_*(\btheta(\{k\}\cup\alpha_n))) \cr
&= \sup_{\btheta(k_{n+1})\in\mathbb{R}^{n+1}} \cp(P^{\{k\}\cup\alpha_{n+1}}_*(\btheta(\{k\}\cup\alpha_{n+1}))), 
\end{align*}
where we use the fact that $Q^{\{k\}\cup\alpha_n,\{k_{n+1}\}}_*(\btheta(\{k\}\cup\alpha_n))$ is a truncation of $P^{\{k\}\cup\alpha_n}_*(\btheta(\{k\}\cup\alpha_n))$. This leads us to 
\begin{align*}
\bar{\theta}^{\{k\}} 
&\le \max\!\left\{ \theta(k)\in\mathbb{R}; \inf_{\btheta(\alpha_{n+1})\in\mathbb{R}^{n+1}} \cp(P^{\{k\}\cup\alpha_{n+1}}_*(\btheta(\{k\}\cup\alpha_{n+1})))^{-1}\le 1 \right\} \cr
&= \max\!\left\{ \theta(k)\in\mathbb{R}; \btheta\in\bar{\Gamma}^{\{k\}\cup\alpha_{n+1}} \right\}, 
\end{align*}
and this completes the proof. 
\end{proof}

By inequality (\ref{eq:nu_limit1}) and Proposition \ref{pr:Rk_cp}, we obtain the following result. 
\begin{theorem} \label{th:nu_limit}
For $k\in D$ and $\alpha\in\scrP(D\setminus\{k\})$ and for $\bx(\{k\}^C)\in\mathbb{Z}_+^{d-1}$ and $j\in S_0$, 
\begin{equation}
\liminf_{n\to\infty} \frac{1}{n} \log \nu_{((x(k)=n,\bx(\{k\}^C)),j)} \ge - \max\!\left\{ \theta(k); \btheta\in\bar{\Gamma}^{\{k\}\cup\alpha} \right\}. 
\label{eq:nu_limit2}
\end{equation}
\end{theorem}

\begin{remark}
From the proof of Proposition \ref{pr:Rk_cp}, we can see that, for $k\in D$ and for $\alpha, \beta\in\scrP(D\setminus\{k\})$, if $\alpha\subset\beta$,  
\begin{equation}
\max\!\left\{ \theta(k)\in\mathbb{R}; \btheta\in\bar{\Gamma}^{\{k\}\cup\alpha} \right\}
\le \max\!\left\{ \theta(k)\in\mathbb{R}; \btheta\in\bar{\Gamma}^{\{k\}\cup\beta} \right\}. 
\end{equation}
Hence, a smaller index set $\alpha$ gives a tighter bound. However, the value of $d-|\alpha|-1$ is the dimension of the Markov chain governed by $P^{\{k\}\cup\alpha}_*(\bzero_{|\alpha|})$, and it becomes harder to estimate the convergence parameter of $P^{\{k\}\cup\alpha}_*(\btheta(\{k\}\cup\alpha))$ as the cardinality of $\alpha$ becomes smaller. 
When $\{k\}\cup\alpha=D$, the dimension of the matrix $P^{D}_*(\btheta)$ is finite and its convergence parameter, which is the reciprocal of the Perron-Frobenius eigenvalue of $P^{D}_*(\btheta)$, can easily be estimated. 
\end{remark}

%
%
\subsection{Markov additive processes induced by the QBD process}

In Subsection \ref{sec:MGsolution}, we considered one-dimensional MA-processes, $\{\tilde{\bY}^{\{k\}}_n\}$, $k\in D$, induced by the original $d$-dimensional QBD process $\{\bY_n\}=\{(\bX_n,J_n)\}$. In this subsection, we consider such MA-processes comprehensively and give an interpretation to the lower bounds obtained in the previous subsection. 

For a nonempty set $\alpha\in\scrP(D)$, let $\{\tilde{\bY}^{\alpha}_n\}=\{(\tilde{\bX}^{\alpha}_n(\alpha),(\tilde{\bX}^{\alpha}_n(\alpha^C),\tilde{J}^{\alpha}_n))\}$ be a $|\alpha|$-dimensional MA-process on the state space $\tilde{\mathbb{S}}^{\alpha}=\mathbb{Z}^{|\alpha|}\times (\mathbb{Z}_+^{d-|\alpha|}\times S_0)$, generated from the $d$-dimensional QBD process $\{\bY_n\}=\{(\bX_n,J_n)\}$ by removing the reflecting boundaries on $\mathbb{B}^\beta$ for $\beta\in\scrP(\alpha)$ and extending the $d$-dimensional QBD process over the negative region with respect to $\bX_n(\alpha)$. 
The behavior  of the MA-process $\{\tilde{\bY}^{\alpha}_n\}$  is stochastically identical to $\{\bY_n\}=\{(\bX_n,J_n)\}$ when $\bX_n(\alpha)>1$.  In terminology of Markov additive process, the level state $\tilde{\bX}^{\alpha}_n(\alpha)$ is the additive part and the phase state $(\tilde{\bX}^{\alpha}_n(\alpha^C),\tilde{J}^{\alpha}_n)$ is the background state. The background process $\{(\tilde{\bX}^{\alpha}_n(\alpha^C),\tilde{J}^{\alpha}_n)\}$ is a Markov chain itself. 
The MA-kernel of the MA-process is given by 
\begin{equation}
\left\{ \tilde{A}^{\alpha}_{\bl(\alpha)},\ \bl(\alpha)\in\{-1,0,1\}^{|\alpha|} \right\}, 
\end{equation}
where 
\[
\tilde{A}^{\alpha}_{\bl(\alpha)} = 
\left( P_{\bx,\bx'}; \bx(\alpha)=2,\,\bx'(\alpha)= \bx(\alpha)+\bl(\alpha),\ \bx(\alpha^C),\bx'(\alpha^C)\in\mathbb{Z}_+^{d-|\alpha|} \right). 
\]
From (\ref{eq:Palpha_blockform}) and (\ref{eq:Palphak_def}), the matrix function $P^{\alpha}_*(\btheta(\alpha))$, which is called a Feynman-Kac operator, is given as 
\begin{equation}
P^{\alpha}_*(\btheta(\alpha)) 
= \sum_{\bl(\alpha)\in\{-1,0,1\}^{|\alpha|}}\ e^{\langle \bl(\alpha),\btheta(\alpha) \rangle} \tilde{A}^{\alpha}_{\bl(\alpha)}. 
\end{equation}
Hence, we see that the lower bounds for the directional asymptotic decay rates of the stationary distribution $\bnu$ are given in terms of the convergence parameter of the Feynman-Kac operators for the MA-processes induced by the original $d$-dimensional QBD process.

%
Next, we demonstrate that the lower bounds in Theorem \ref{th:nu_limit} coincide with the directional asymptotic decay rates of the occupancy measure for the MA-processes. 
For a nonempty set $\alpha\in\scrP(D)$, we denote by $\tilde{P}^\alpha$ the transition probability matrix of the MA-process $\{\tilde{\bY}^\alpha_n\}$. The matrix $\tilde{P}^\alpha$ can be represented in block form as 
\[
\tilde{P}^\alpha= \left( \tilde{P}^\alpha_{\bx(\alpha),\bx'(\alpha)}; \bx(\alpha),\bx'(\alpha)\in\mathbb{Z}^{|\alpha|} \right),
\]
where
\begin{equation}
\tilde{P}^\alpha_{\bx(\alpha),\bx'(\alpha)}
= \left\{ \begin{array}{ll}
\tilde{A}^\alpha_{\bx'(\alpha)-\bx(\alpha)}, & \mbox{if $\bx'(\alpha)-\bx(\alpha)\in\{-1,0,1\}^{|\alpha|}$}, \cr
O, & \mbox{otherwise}.
\end{array} \right.
\label{eq:Palpha}
\end{equation}
From this, we can see that $\tilde{P}^\alpha$ has a multiple block tri-diagonal structure without boundaries. 
Let $\tau$ be a stopping time defined as
 \[
 \tau = \inf\!\left\{n\ge 1; \tilde{\bX}^{\alpha}_n(\alpha)\in\mathbb{Z}^{|\alpha|}\setminus(\mathbb{Z}_+\setminus\{0\})^{|\alpha|} \right\}, 
 \]
where $\tilde{\bX}^{\alpha}_n(\alpha)$ is the additive part of the MA-process, and $\{\tilde{H}^\alpha_{\bx(\alpha)},\bx(\alpha)\in\mathbb{Z}_+^{|\alpha|} \}$ the occupancy measure of the MA-process defined as, for $(\bx'(\alpha^C),j')$, $(\bx(\alpha^C),j)\in\mathbb{Z}_+^{d-|\alpha|}\times S_0$, 
\[
[\tilde{H}^\alpha_{\bx(\alpha)}]_{(\bx'(\alpha^C),j'),(\bx(\alpha^C),j)} 
= \mathbb{E}\!\left(\sum_{n=0}^{\tau-1} 1(\tilde{\bY}^\alpha_n=(\bx(\alpha),(\bx(\alpha^C),j)) \,\Big|\, \tilde{\bY}^\alpha_0=(\bzero_{|\alpha|},(\bx'(\alpha^C),j'))    \right). 
\]
Let $\tilde{Q}^\alpha$ be a truncation of $\tilde{P}^\alpha$ defined as 
\[
\tilde{Q}^\alpha= \left( \tilde{P}^\alpha_{\bx(\alpha),\bx'(\alpha)}; \bx(\alpha),\bx'(\alpha)\in\mathbb{Z}_+^{|\alpha|} \right).
\]
We assume the following condition throughout this section. 
\begin{assumption} \label{as:tildeQ_irreducible}
For every nonempty set $\alpha\in\scrP(D)$, the substochastic matrix $\tilde{Q}^\alpha$ is irreducible. 
\end{assumption}

Under this condition, Assumption \ref{as:Qalphak_irreducible} automatically holds. Hence, hereafter, we use Assumption \ref{as:tildeQ_irreducible} instead of Assumption \ref{as:Qalphak_irreducible}. 
The following lemma asserts that the lower bounds in Theorem \ref{th:nu_limit} coincide with the directional asymptotic decay rates of the occupancy measure. 
\begin{lemma} \label{le:Hx_asymp}
For a nonempty set $\alpha\in\scrP(D)$ and $k\in\alpha$ and for $\bx(\alpha\setminus\{k\})\in\mathbb{Z}_+^{|\alpha|-1}$ and $(\bx(\alpha^C),j),\,(\bx'(\alpha^C),j')\in\mathbb{Z}_+^{d-|\alpha|}\times S_0$, 
\begin{equation}
\lim_{n\to\infty} \frac{1}{n} \log [\tilde{H}^\alpha_{(x(k)=n,\bx(\alpha\setminus\{k\}))}]_{(\bx'(\alpha^C),j'),(\bx(\alpha^C),j)} 
= - \max\!\left\{\theta(k); \btheta\in\bar{\Gamma}^{\alpha} \right\}. 
\label{eq:Hx_limit1}
\end{equation}
\end{lemma}
\begin{proof}
For a nonempty set $\alpha\in\scrP(D)$, consider the MA-process $\{\tilde{\bY}^\alpha_n\}=\{(\tilde{\bX}^\alpha_n(\alpha),(\tilde{\bX}^\alpha_n(\alpha^C),\tilde{J}^\alpha_n))\}$. Without loss of generality, we assume $\alpha=\{k_1,k_2,...,k_{n_0}\}$, where $n_0=|\alpha|$. 
Regarding $\tilde{X}^\alpha_n(k_1)$ as a level state and $((\tilde{\bX}^\alpha(\alpha\setminus\{k_1\}),\tilde{\bX}^\alpha(\alpha^C)),\tilde{J}^\alpha_n)$ as a phase state, we obtain the following block tri-diagonal representation of $\tilde{Q}^\alpha$:
\[
\tilde{Q}^\alpha
= \begin{pmatrix}
\tilde{Q}^{\alpha,\{k_1\}}_0 & \tilde{Q}^{\alpha,\{k_1\}}_1 & & \cr
\tilde{Q}^{\alpha,\{k_1\}}_{-1} & \tilde{Q}^{\alpha,\{k_1\}}_0 & \tilde{Q}^{\alpha,\{k_1\}}_1 & & \cr
& \ddots & \ddots & \ddots 
\end{pmatrix},
\]
where, for $l\in\{-1,0,1\}$, 
\[
\tilde{Q}^{\alpha,\{k_1\}}_l = \left( \tilde{P}^\alpha_{\bx(\alpha),\bx'(\alpha)}; x(k_1)=0,\,x'(k_1)=l,\,\bx(\alpha\setminus\{k_1\}),\bx'(\alpha\setminus\{k_1\})\in\mathbb{Z}_+^{|\alpha|-1} \right) 
\] 
and we use the fact that the level process $\{\tilde{X}^\alpha_n(k_1)\}$ is skip free. 
Let $\tilde{R}^{\alpha,\{k_1\}}$ be the rate matrix generated from the triplet $\{ \tilde{Q}^{\alpha,\{k_1\}}_{-1},\tilde{Q}^{\alpha,\{k_1\}}_0,\tilde{Q}^{\alpha,\{k_1\}}_1 \}$ and $\tilde{Q}^{\alpha,\{k_1\}}_*(\theta(k_1))$ be the matrix function defined as 
\[
\tilde{Q}^{\alpha,\{k_1\}}_*(\theta(k_1))
= e^{-\theta(k_1)} \tilde{Q}^{\alpha,\{k_1\}}_{-1}+\tilde{Q}^{\alpha,\{k_1\}}_0+ e^{\theta(k_1)} \tilde{Q}^{\alpha,\{k_1\}}_1. 
\]
By Lemma \ref{le:RandG_cp}, we have
\begin{equation}
\log\cp(\tilde{R}^{\alpha,\{k_1\}}) = \bar{\theta}^{\alpha,\{k_1\}} = \max\left\{ \theta(k_1)\in\mathbb{R}; \cp(\tilde{Q}^{\alpha,\{k_1\}}_*(\theta(k_1)))^{-1}\le 1 \right\}.
\label{eq:bartheta_alphak1}
\end{equation}
Furthermore, we have, for $\bx(\alpha\setminus\{k_1\})\in\mathbb{Z}_+^{|\alpha|-1}$ and for $(\bx(\alpha^C),j),\,(\bx'(\alpha^C),j')\in\mathbb{Z}_+^{d-|\alpha|}\times S_0$, 
\[
[\tilde{H}^\alpha_{(x(k_1)=n,\bx(\alpha\setminus\{k_1\}))}]_{(\bx'(\alpha^C),j'),(\bx(\alpha^C),j)} 
= [(\tilde{R}^{\alpha,\{k_1\}})^n]_{((\bx'(\alpha\setminus\{k_1\})=0,\bx'(\alpha^C)),j'),((\bx(\alpha\setminus\{k_1\}),\bx(\alpha^C)),j)}.
\]
Hence, by Corollary \ref{co:cpRlimit}, if the $(\bx'(\alpha^C),j')$-row of $\tilde{Q}_1^{\alpha,\{k_1\}}$ contains at least one positive element, we have
\begin{equation}
\lim_{n\to\infty} \frac{1}{n} \log [\tilde{H}^\alpha_{(x(k_1)=n,\bx(\alpha\setminus\{k_1\}))}]_{(\bx'(\alpha^C),j'),(\bx(\alpha^C),j)} 
= - \bar{\theta}^{\alpha,\{k_1\}}. 
\label{eq:Halpha_limit}
\end{equation}
Furthermore, regarding $\tilde{X}^\alpha_n(k_2)$ as a level state, the matrix function $\tilde{Q}^{\alpha,\{k_1\}}_*(\theta(k_1))$ is also represented in block tri-diagonal form as
\[
\tilde{Q}^\alpha_*(\theta(k_1))
= \begin{pmatrix}
\tilde{Q}^{\alpha,\{k_1,k_2\}}_{*,0}(\theta(k_1)) & \tilde{Q}^{\alpha,\{k_1,k_2\}}_{*,1}(\theta(k_1)) & & \cr
\tilde{Q}^{\alpha,\{k_1,k_2\}}_{*,-1}(\theta(k_1)) & \tilde{Q}^{\alpha,\{k_1,k_2\}}_{*,0}(\theta(k_1)) & \tilde{Q}^{\alpha,\{k_1,k_2\}}_{*,0}(\theta(k_1)) & & \cr
& \ddots & \ddots & \ddots 
\end{pmatrix},
\]
where, for $l\in\{-1,0,1\}$, 
\begin{align*}
\tilde{Q}^{\alpha,\{k_1,k_2\}}_{*,l}(\theta(k_1)) 
&= \biggl( \sum_{m(k_1)\in\{-1,0,1\}} e^{m(k_1) \theta(k_1)} \tilde{P}^\alpha_{(x(k_1)=0,\bx(\alpha\setminus\{k_1\})),(x'(k_1)=m(k_1),\bx'(\alpha\setminus\{k_1\}))}; \cr
&\qquad\qquad x(k_2)=0,\,x'(k_2)=l,\,\bx(\alpha\setminus\{k_1,k_2\}),\bx'(\alpha\setminus\{k_1,k_2\})\in\mathbb{Z}_+^{|\alpha|-2} \biggr). 
\end{align*}
Define a matrix function $\tilde{Q}^{\alpha,\{k_1,k_2\}}_*(\btheta(k_1,k_2))$ as 
\begin{align*}
&\tilde{Q}^{\alpha,\{k_1,k_2\}}_*(\btheta(k_1,k_2)) \cr
&\quad= e^{-\theta(k_2)} \tilde{Q}^{\alpha,\{k_1,k_2\}}_{*,-1}(\theta(k_1)) + \tilde{Q}^{\alpha,\{k_1,k_2\}}_{*,0}(\theta(k_1)) + e^{\theta(k_2)} \tilde{Q}^{\alpha,\{k_1,k_2\}}_{*,0}(\theta(k_1)) \cr
&\quad= \biggl( \sum_{\bm(k_1,k_2)\in\{-1,0,1\}^2} e^{\langle \bm(k_1,k_2), \btheta(k_1,k_2) \rangle} \tilde{P}^\alpha_{(\bx(k_1,k_2)=0,\bx(\alpha\setminus\{k_1,k_2\})),(\bx'(k_1,k_2)=\bm(k_1,k_2),\bx'(\alpha\setminus\{k_1,k_2\}))}; \cr
&\qquad\qquad\qquad \bx(\alpha\setminus\{k_1,k_2\}),\bx'(\alpha\setminus\{k_1,k_2\})\in\mathbb{Z}_+^{|\alpha|-2} \biggr). 
\end{align*}
Then, by Lemma \ref{le:Q_cp}, we obtain 
\[
\cp(\tilde{Q}^{\alpha,\{k_1\}}_*(\theta(k_1)))
= \sup_{\theta(k_2)\in\mathbb{R}} \cp(\tilde{Q}^{\alpha,\{k_1,k_2\}}_*(\btheta(k_1,k_2)))
\]
and this and (\ref{eq:bartheta_alphak1}) lead us to
\begin{equation}
\bar{\theta}^{\alpha,\{k_1\}}
= \max\left\{ \theta(k_1)\in\mathbb{R}; \inf_{\theta(k_2)\in\mathbb{R}} \cp(\tilde{Q}^{\alpha,\{k_1,k_2\}}_*(\btheta(k_1,k_2)))^{-1}\le 1 \right\}.
\label{eq:bartheta_alphak2}
\end{equation}
Repeating this procedure $n_0-2$ more times, we obtain 
\begin{equation}
\bar{\theta}^{\alpha,\{k_1\}}
= \max\left\{ \theta(k_1)\in\mathbb{R}; \inf_{\btheta(\alpha\setminus\{k_1\})\in\mathbb{R}^{n_0-1}} \cp(\tilde{Q}^{\alpha,\alpha}_*(\btheta(\alpha)))^{-1}\le 1 \right\},
\label{eq:bartheta_alphakn0}
\end{equation}
where $\alpha=\{k_1,k_2,...,k_{n_0}\}$ and 
\begin{align*}
\tilde{Q}^{\alpha,\alpha}_*(\btheta(\alpha)) 
&= \sum_{\bm(\alpha)\in\{-1,0,1\}^{n_0}} e^{\langle \bm(\alpha),\btheta(\alpha) \rangle} \tilde{P}^\alpha_{\bzero_{n_0},\bm(\alpha)} 
=P^\alpha_*(\btheta(\alpha)). 
\end{align*}
This completes the proof. 
\end{proof}

%
%
\section{Concluding remark} \label{sec:conclusion}

It is a natural thought that the asymptotic decay rates of the stationary distribution for a $d$-dimensional QBD process are determined by the MA-processes induced by it. The following is a candidate. 
\begin{conjecture} \label{co:nu_limit}
For $k\in D$, $\bx(\{k\}^C)\in\mathbb{Z}_+^{d-1}$ and $j\in S_0$, 
\begin{equation}
\liminf_{n\to\infty} \frac{1}{n} \log \nu_{((x(k)=n,\bx(\{k\}^C)),j)} = - \sup\!\bigg\{ s(k); \bs\in \bigcap_{\alpha\in\scrP(D),\alpha\ne\emptyset} \bar{\Gamma}^{\alpha} \bigg\}. 
\label{eq:nu_limit3}
\end{equation}
\end{conjecture}
This conjecture holds true when $d=2$, see Proposition 3.1 and Lemma 3.1 of \cite{Ozawa18}.

%
%

%
%
\appendix
%
\section{Convexity of the reciprocal of a convergence parameter} \label{sec:cp_convex}

Let $n$ be a positive integer and $\bx=(x_1,x_2,...,x_n)\in\mathbb{R}^n$. 
We say that a positive function $f(\bx)$ is log-convex in $\bx$ if $\log f(\bx)$ is convex in $\bx$, and denote by $\mathfrak{S}_n$ the class of all log-convex functions of $n$ variables, together with the function identically zero. Note that, $\mathfrak{S}_n$ is closed under addition, multiplication, raising to any positive power, and ``$\limsup$" operation. Furthermore, a log-convex function is a convex function. 

Let $F(\bx)=(f_{ij}(\bx),i,j\in\mathbb{Z}_+)$ be a matrix function each of whose elements belongs to the class $\mathfrak{S}_n$, i.e., for every $i,j\in\mathbb{Z}_+$, $f_{i,j}\in\mathfrak{S}_n$. In \cite{Kingman61}, it is demonstrated that when $n=1$ and $F(x)$ is a square matrix of a finite dimension, the maximum eigenvalue of $F(x)$ is a log-convex function in $x$. Analogously, we obtain the following lemma. 
\begin{lemma} \label{le:cp_convex}
For every $\bx\in\mathbb{R}^n$, assume all iterates of $F(\bx)$ is finite and $F(\bx)$ is irreducible. Then, the reciprocal of the convergence parameter of $F(\bx)$, $\cp(F(\bx))^{-1}$, is log-convex in $\bx$ or identically zero.
\end{lemma}
\begin{proof}
For $k\ge 0$, we denote by $f^{(k)}_{i,j}(\bx)$ the $(i,j)$-element of $F(\bx)^k$. 
First, we show that, for any $k\ge 1$ and for any $i,j\in\mathbb{Z}_+$, $f^{(k)}_{i,j}(\bx)\in\mathfrak{S}_n$. It is obvious when $k=1$. Suppose that it holds for $k$. Then, we have, for any $i,j\in\mathbb{Z}_+$,  
\begin{equation}
f^{(k+1)}_{i,j}(\bx) = \lim_{m\to\infty} \sum_{l=0}^m f^{(k)}_{i,l}(\bx)\, f_{l,j}(\bx), 
\end{equation}
and this leads us to $f^{(k+1)}_{i,j}(\bx)\in\mathfrak{S}_n$ since $\mathfrak{S}_n$ is closed under addition, multiplication and ``$\limsup$" (``$\lim$") operation. Therefore, for any $k\ge 1$, every element of $F(\bx)^n$ belongs to $\mathfrak{S}_n$. 

Next, we note that, by Theorem 6.1 of \cite{Seneta06}, since $F(\bx)$ is irreducible, all elements of the power series $\sum_{k=0}^\infty z^k F(\bx)^k$ have the common convergence radius (convergence parameter), which is denoted by $\cp(F(\bx))$. 
By Cauchy-Hadamard theorem, we have, for any $i,j\in\mathbb{Z}_+$, 
\begin{equation}
\cp(F(\bx))^{-1} = \limsup_{k\to\infty} \bigl(f^{(k)}_{i,j}(\bx)\bigr)^{1/k}, 
\end{equation}
and this implies $\cp(F(\bx))^{-1}\in\mathfrak{S}_n$ since $\bigl(f^{(k)}_{i,j}(\bx)\bigr)^{1/k}\in\mathfrak{S}_n$ for any $k\ge 1$. 
\end{proof}

%
\section{Proof of Lemma \ref{le:RandGmatrix_equations}} \label{sec:proof_RandGmatrix_equations}

\begin{proof}
{\it (i)}\quad 
For $n\ge 1$, $\scrI_{D,1,n}$ and $\scrI_{U,1,n}$ satisfy 
\begin{align*}
\scrI_{D,1,n} &= \biggl\{\bi_{(n)}\in\{-1,0,1\}^n:\ \sum_{l=1}^k i_l\ge 0\ \mbox{for $k\in\{1,2,...,n-2\}$},\ \sum_{l=1}^{n-1} i_l=0\ \mbox{and}\ i_n=-1 \biggr\} \cr
&= \{(\bi_{(n-1)},-1):\ \bi_{(n-1)}\in\scrI_{n-1} \}, \\
\scrI_{U,1,n} &= \biggl\{\bi_{(n)}\in\{-1,0,1\}^n:\ i_1=1,\,\sum_{l=2}^k i_l\ge 0\ \mbox{for $k\in\{2,...,n-1\}$}\ \mbox{and} \sum_{l=2}^n i_l=0 \biggr\} \cr 
&= \{(1,\bi_{(n-1)}):\ \bi_{(n-1)}\in\scrI_{n-1} \},
\end{align*}
where $\bi_{(n)}=(i_1,i_2,...,i_n)$. Hence, by Fubini's theorem, we have, for $i,j\in\mathbb{Z}_+$, 
\begin{align*}
&[G]_{i,j} = \sum_{n=1}^\infty \sum_{k=0}^\infty [Q_{0,0}^{(n-1)}]_{i,k}\, [ A_{-1}]_{k,j}  = [N  A_{-1}]_{i,j}, \\
&[R]_{i,j} = \sum_{n=1}^\infty \sum_{k=0}^\infty [ A_1]_{i,k} [Q_{0,0}^{(n-1)}]_{k,j} = [ A_1 N]_{i,j}. 
\end{align*}. 

{\it (ii)}\quad 
We prove equation (\ref{eq:Gmatrix_equation0}). 
In a manner similar to that used in (i), we have, for $n\ge 3$, 
\begin{align*}
\scrI_{D,1,n} 
&= \{(0,\bi_{(n-1)}):\ \bi_{(n-1)}\in\scrI_{D,1,n-1} \} \cup \{(1,\bi_{(n-1)}):\ \bi_{(n-1)}\in\scrI_{D,2,n-1} \}, \\
%
%
\scrI_{D,2,n} 
&= \bigcup_{m=1}^{n-1} \{(\bi_{(m)},\bi_{(n-m)}):\ \bi_{(m)}\in\scrI_{D,1,m}\ \mbox{and}\ \bi_{(n-m)}\in\scrI_{D,1,m-n} \}. 
\end{align*}
Hence, we have, for $n\ge 3$, 
\begin{align*}
D^{(n)} 
&= A_0 D^{(n-1)} + A_1 \sum_{\bi_{(n-1)}\in\scrI_{D,2,n-1}} A_{i_1} A_{i_2} \cdots A_{i_{n-1}} \cr
&= A_0 D^{(n-1)} + A_1 \sum_{m=1}^{n-1} D^{(m)} D^{(n-m-1)}, 
\end{align*} 
and by Fubini's theorem, we obtain, for $i,j\in\mathbb{Z}_+$, 
\begin{align*}
[G]_{i,j} 
&= [ D^{(1)}]_{i,j} + \sum_{n=2}^\infty \sum_{k=0}^\infty [ A_0]_{i,k} [D^{(n-1)}]_{k,j} 
+ \sum_{n=3}^\infty \sum_{m=1}^{n-2} \sum_{k=0}^\infty \sum_{l=0}^\infty [ A_1]_{i,k} [^{m} D^{(m)}]_{k,l} [D^{(n-m-1)}]_{l,j} \cr
&= [ A_{-1}]_{i,j} + [ A_0 G]_{i,j} + [ A_1 G^2]_{i,j}, 
\end{align*}
where we use the fact that $D^{(1)}=A_{1}$ and $D^{(2)}=A_0 A_{-1}=A_0 D^{(1)}$. 
Equation (\ref{eq:Rmatrix_equation0}) is analogously proved. 

{\it (iii)}\quad 
We prove equation (\ref{eq:NandH_relation}). 
In a manner similar to that used in (i), we have, for $n\ge 1$, 
\begin{align*}
\scrI_{n} 
%
%
&= \{(0,\bi_{(n-1)}):\ \bi_{(n-1)}\in\scrI_{n-1} \} \cr
&\qquad \cup \left(\cup_{m=2}^n \{(1,\bi_{(m-1)},\bi_{(n-m)}):\ \bi_{(m-1)}\in\scrI_{D,1,m-1}, \bi_{(n-m)}\in\scrI_{n-m} \} \right), \\
\scrI_{n} &= \{(\bi_{(n-1)},0):\ \bi_{(n-1)}\in\scrI_{n-1} \} \cr
&\qquad \cup \left(\cup_{m=0}^{n-2} \{(\bi_{(m)},1,\bi_{(n-m-1)}):\ \bi_{(m)}\in\scrI_{m}, \bi_{(n-m-1)}\in\scrI_{D,1,n-m-1} \} \right). 
\end{align*}
Hence, we have, for $n\ge 1$, 
\begin{align*}
&Q_{0,0}^{(n)} = A_0 Q_{0,0}^{(n-1)} + \sum_{m=2}^n A_1 D^{(m-1)} Q_{0,0}^{(n-m)}, \\
&Q_{0,0}^{(n)} = Q_{0,0}^{(n-1)} A_0 + \sum_{m=0}^{n-2} Q_{0,0}^{(m)} A_1 D^{(m-n-1)} , 
\end{align*}
and by Fubini's theorem, we obtain, for $i,j\in\mathbb{Z}_+$, 
\begin{align*}
[N]_{i,j} 
&= \delta_{ij} 
+ \sum_{n=1}^\infty \sum_{k=0}^\infty [A_0]_{i,k} [Q_{0,0}^{(n-1)}]_{k,j} 
+ \sum_{n=2}^\infty \sum_{m=2}^n \sum_{k=0}^\infty \sum_{l=0}^\infty [A_1]_{i,k} [D^{(m-1)}]_{k,l} [Q_{0,0}^{(n-m)}]_{l,j} \cr
&= \delta_{i,j} + [A_0 N]_{i,j} + [A_1 G N]_{i,j}, \\
[N]_{i,j} 
&= \delta_{ij} 
+ \sum_{n=1}^\infty \sum_{k=0}^\infty [Q_{0,0}^{(n-1)}]_{i,k} [A_0]_{k,j} 
+ \sum_{n=2}^\infty \sum_{m=0}^{n-2} \sum_{k=0}^\infty \sum_{l=0}^\infty [Q_{0,0}^{(m)}]_{i,k} [A_1]_{k,l} [D^{(n-m-1)}]_{l,j}  \cr
&= \delta_{i,j} + [N A_0]_{i,j} + [N A_1 G]_{i,j},
\end{align*}
This leads us to equation (\ref{eq:NandH_relation}). 
\end{proof}

%
\section{A sufficient condition ensuring $\chi(\theta)$ is unbounded} \label{sec:chi_unbounded}

\begin{proposition} \label{pr:chi_unbounded}
Assume $P$ is irreducible, then $\chi(\theta)$ is unbounded in both the directions, i.e., $\lim_{\theta\to -\infty} \chi(\theta) = \lim_{\theta\to\infty} \chi(\theta)=\infty$. 
\end{proposition}

\begin{proof}
Note that, since $P$ is irreducible, $A_*$ is also irreducible. 
For $n\in\mathbb{N}$, $j\in\mathbb{Z}_+$ and $\theta\in\mathbb{R}$, $A_*(\theta)^n$ satisfies 
\begin{align}
&[ A_*(\theta)^n ]_{jj}
=\ \sum_{\bi_{(n)}\in\{-1,0,1\}^n} [ A_{i_1} A_{i_2} \times \cdots \times A_{i_n} ]_{jj}\,e^{\theta \sum_{k=1}^n i_k},
\label{eq:Asesn}
\end{align}
where $\bi_{(n)}=(i_1,i_2,...,i_n)$.
Since $P$ is irreducible, there exist $n_0>1$ and $\bi_{(n_0)}\in\{-1,0,1\}^{n_0}$ such that $[ A_{i_1} A_{i_2} \times \cdots \times A_{i_{n_0}} ]_{jj}>0$ and $\sum_{k=1}^{n_0} i_k =1$. For such a $n_0$, we have $[A_*(\theta)^{n_0}]_{jj}\ge c e^\theta$ for some $c>0$. 
This implies that, for any $m\in\mathbb{N}$, $[A_*(\theta)^{n_0 m}]_{jj}\ge c^m e^{m \theta}$ and we have 
\[
\chi(\theta)
= \limsup_{m\to\infty} ([A_*(\theta)^m]_{jj})^{\frac{1}{m}} 
\ge \limsup_{m\to\infty} ([A_*(\theta)^{n_0 m}]_{jj})^{\frac{1}{n_0 m}}
\ge c^{\frac{1}{n_0}} e^{\frac{\theta}{n_0}}. 
\]
Therefore, $\lim_{\theta\to\infty} \chi(\theta)=\infty$. 
Analogously, we can obtain that $\chi(\theta)\ge c^{\frac{1}{n_0}} e^{-\frac{\theta}{n_0}}$ for some $n_0\in\mathbb{N}$ and $c>0$, and this implies that $\lim_{\theta\to -\infty} \chi(\theta)=\infty$. 
\end{proof}

%
\section{Proof of Proposition \ref{pr:R_u} and Corollary \ref{co:cpRlimit}} \label{sec:proof_R_u}

%
%
\begin{proof}[Proof of Proposition \ref{pr:R_u}]
Let $S_1$ be the set of indexes of nonzero rows of $A_1$, i.e., $S_1=\{k\in\mathbb{Z}_+; \mbox{the $k$-th row of $A_1$ is nonzero} \}$, and $S_2=\mathbb{Z}_+\setminus S_1$.  For $i\in\{-1,0,1\}$, reorder the rows and columns of $A_i$ so that it is represented as
\[
A_i =\begin{pmatrix}
A_{i,11} & A_{i,12} \cr
A_{i,21} & A_{i,22}
\end{pmatrix},
\]
where $A_{i,11}=( [A_i]_{k,l}; k,l\in S_1 )$, $A_{i,12}=( [A_i]_{k,l}; k\in S_1, l\in S_2 )$, $A_{i,21}=( [A_i]_{k,l}; k\in S_2, l\in S_1 )$ and $A_{i,22}=( [A_i]_{k,l}; k,l\in S_2 )$. By the definition of $S_1$, every row of $A_{1,11}+A_{1,12}$ is nonzero and $A_{1,21}=A_{1,22}=O$. 
Since $Q$ is irreducible and $R$ is finite, $N$ is also finite and positive. Hence, $R$ is given as
\begin{equation}
R = A_1 N 
= \begin{pmatrix}
R_{11} & R_{12} \cr
O & O
\end{pmatrix},
\label{eq:R11R12}
\end{equation}
where $R_{11}=( [R]_{k,l}; k,l\in S_1 )$ is positive and hence irreducible; $R_{12}=( [R]_{k,l}; k\in S_1, l\in S_2 )$ is also positive. 
Since $R_{11}$ is a submatrix of $R$, we have $\cp(R_{11})\ge \cp(R)$. 

We derive an inequality with respect to $R_{11}$ and $R_{12}$. From (\ref{eq:Rmatrix_equation0}), we obtain $R\ge R^2 A_{-1} + R A_0$ and, from this inequality, 
\begin{align}
&R_{11} \ge R_{11} R_{12} A_{-1,21} + R_{12} A_{0,21} \label{eq:R11_ineq} \\
&R_{12} \ge R_{11} R_{12} A_{-1,22} + R_{12} A_{0,22}. \label{eq:R12_ineq}
\end{align}
For $n\ge 1$ and $\bi_{(n)}=(i_1,i_2,...,i_n)\in\{-1,0\}^n$, define $A_{\bi_{(n)},22}$ and $\Vert \bi_{(n)} \Vert$ as
\[
A_{\bi_{(n)},22} = A_{i_n,22}\times A_{i_{n-1},22}\times \cdots \times A_{i_1,22},\quad 
\Vert \bi_{(n)} \Vert = \sum_{k=1}^n |i_k|.
\] 
Then, by induction using (\ref{eq:R12_ineq}), we obtain, for $n\ge 1$, 
\begin{equation}
R_{12} \ge \sum_{\bi_{(n)}\in\{-1,0\}^n} R_{11}^{\Vert \bi_{(n)} \Vert} R_{12} A_{\bi_{(n)},22}, 
\label{eq:R12_ineq2}
\end{equation}
and this and (\ref{eq:R11_ineq}) lead us to, for $n\ge 1$,  
\begin{equation}
R_{11} \ge \sum_{\bi_{(n)}\in\{-1,0\}^n} R_{11}^{\Vert \bi_{(n)} \Vert} R_{12} A_{\bi_{(n-1)},22} A_{i_n,21}, 
\label{eq:R11_ineq2}
\end{equation}
where $A_{\bi_{(0)},22} = I$. 
We note that since $A_*=A_{-1}+A_0+A_1$ is irreducible and $A_{1,21}=A_{1,22}=O$, for every $k\in S_2$ and every $l\in S_1$, there exist $n_0\ge 1$ and $\bi_{(n_0)}\in\{-1,0\}^{n_0}$ such that $[A_{\bi_{(n_0-1)},22} A_{i_{n_0},21}]_{k,l}>0$. 

Let $\alpha$ be the convergence parameter of $R_{11}$. Since $R_{11}$ is irreducible, $R_{11}$ is either $\alpha$-recurrent or $\alpha$-transient. 
First, we assume $R_{11}$ is $\alpha$-recurrent. Then, there exists a positive vector $\bu_1$ such that $\alpha \bu_1 R_{11}=\bu_1$. If $\bu_1 R_{12}<\infty$, then $\bu=(\bu_1, \alpha \bu_1 R_{12})$ satisfies $\alpha \bu R = \bu$ and we obtain $\cp(R)\ge\alpha=\cp(R_{11})$. Since $\cp(R)\le\cp(R_{11})$, this implies $\alpha=\cp(R)=e^{\bar{\theta}}$ and we obtain statement (i) of the proposition. We, therefore, prove $\bu_1 R_{12}<\infty$. 
Suppose, for some $k\in S_2$, the $k$-th element of $\bu_1 R_{12}$ diverges. For this $k$ and any $l\in S_1$, there exist $n_0\ge 1$ and $\bi_{(n_0)}\in\{-1,0\}^{n_0}$ such that $[A_{\bi_{(n_0-1)},22} A_{i_{n_0},21}]_{k,l}>0$. Hence, from (\ref{eq:R11_ineq2}), we obtain 
\begin{align*}
[\alpha^{-1} \bu_1]_l 
= [\bu_1 R_{11}]_l 
&\ge [\bu_1 R_{11}^{\Vert \bi_{(n_0)} \Vert} R_{12}]_k [A_{\bi_{(n_0-1)},22} A_{i_{n_0},21}]_{k,l}  \cr
&= \alpha^{-\Vert \bi_{(n_0)} \Vert} [\bu_1 R_{12}]_k [A_{\bi_{(n_0-1)},22} A_{i_{(n_0)},21}]_{k,l}.
\end{align*}
This contradicts $\bu_1$ is finite and we see $\bu_1 R_{12}$ is finite. 

Next, we assume $R_{11}$ is $\alpha$-transient, i.e., $\sum_{n=0}^\infty \alpha^n R_{11}^n<\infty$. We have 
\[
\sum_{n=0}^\infty \alpha^n R^n 
=\begin{pmatrix}
\sum_{n=0}^\infty \alpha^n R_{11}^n & \sum_{n=1}^\infty \alpha^n R_{11}^{n-1} R_{12} \cr
O & O
\end{pmatrix}. 
\]
Hence, in order to prove $\sum_{n=0}^\infty \alpha^n R^n<\infty$, it suffices to demonstrate $\sum_{n=1}^\infty \alpha^n R_{11}^{n-1} R_{12}<\infty$. 
Suppose, for some $k\in S_1$ and some $l\in S_2$, the $(k,l)$-element of $\sum_{n=1}^\infty \alpha^n R_{11}^{n-1} R_{12}$ diverges. For this $l$ and any $m\in S_1$, there exist $n_0\ge 1$ and $\bi_{(n_0)}\in\{-1,0\}^{n_0}$ such that $[A_{\bi_{(n_0-1)},22} A_{i_{n_0},21}]_{l,m}>0$. 
For such an $n_0$ and $\bi_{(n_0)}$, we obtain from (\ref{eq:R11_ineq2}) that, for $n\ge 1$,  
\[
R_{11}^n = R_{11}^{n-1} R_{11} \ge R_{11}^{\Vert \bi_{(n_0)} \Vert} R_{11}^{n-1} R_{12} A_{\bi_{(n_0-1)},22} A_{i_{n_0},21}, 
\]
where every diagonal element of $R_{11}^{\Vert \bi_{(n_0)} \Vert}$ is positive. From this inequality, we obtain 
\[
\left[\sum_{n=1}^\infty \alpha^n R_{11}^n\right]_{k,m} 
\ge \left[ R_{11}^{\Vert \bi_{(n_0)} \Vert} \right]_{k,k} \left[ \sum_{n=1}^\infty \alpha^n R_{11}^{n-1} R_{12} \right]_{k,l} \left[ A_{\bi_{(n_0-1)},22} A_{i_{n_0},21} \right]_{l,m}.
\]
This contradicts $R_{11}$ is $\alpha$-transient and we obtain $\sum_{n=0}^\infty \alpha^n R^n<\infty$. Furthermore, this leads us to $\cp(R)\ge \alpha=\cp(R_{11})$ and, from this and $\cp(R)\le\cp(R_{11})$, we see $\alpha=\cp(R)=e^{\bar{\theta}}$. As a result, we obtain statement (ii) of the proposition and this completes the proof. 
\end{proof}

%
%
\begin{proof}[Proof of Corollary \ref{co:cpRlimit}]
In a manner similar to that used in the proof of Proposition \ref{pr:R_u}, letting $S_1$ be the set of indexes of nonzero rows of $A_1$ and $S_2=\mathbb{Z}_+\setminus S_1$, and reordering the rows and columns of $R$  according to $S_1$ and $S_2$, we obtain $R$ given by expression (\ref{eq:R11R12}), where $R_{11}=( [R]_{k,l}; k,l\in S_1 )$ is positive and hence irreducible and $R_{12}=( [R]_{k,l}; k\in S_1, l\in S_2 )$ is also positive. From the proof of Proposition \ref{pr:R_u},  we know that $ \cp(R)=\cp(R_{11})=e^{\bar{\theta}}$. 
By these facts and the Cauchy-Hadamard theorem, we obtain, for $i,j\in S_1$, $k\in S_2$ and $n\ge 1$, 
\begin{align}
&\limsup_{n\to\infty} \left( [R^n]_{i,j} \right)^{\frac{1}{n}}
= \limsup_{n\to\infty} \left( [R_{11}^n]_{i,j} \right)^{\frac{1}{n}}
= e^{-\bar{\theta}}, \label{eq:limsupR11} \\
&\limsup_{n\to\infty} \left( [R^n]_{i,k} \right)^{\frac{1}{n}} \le e^{-\bar{\theta}}.
\end{align}
Since $[R_{11}^n]_{i,i}$ is subadditive with respect to $n$, i.e., $[R_{11}^{n_1+n_2}]_{i,i}\ge [R_{11}^{n_1}]_{i,i}\, [R_{11}^{n_2}]_{i,i}$ for $n_1,n_2\in\mathbb{Z}_+$, the limit sup in equation (\ref{eq:limsupR11}) can be replaced by the limit for $i=j$ (see, e.g., Lemma A.4 of \cite{Seneta06}). 
Furthermore, we have, for $i,j\in S_1$, $k\in S_2$ and $n\ge 1$, 
\begin{align}
&\liminf_{n\to\infty} \left( [R^n]_{i,j} \right)^{\frac{1}{n}} 
= \liminf_{n\to\infty} \left( [R_{11}^{n-1} R_{11}]_{i,j} \right)^{\frac{1}{n}} 
\ge \liminf_{n\to\infty} \left( [R_{11}^{n-1}]_{i,i} [R_{11}]_{i,j} \right)^{\frac{1}{n}} 
= e^{-\bar{\theta}}, \\
&\liminf_{n\to\infty} \left( [R^n]_{i,k} \right)^{\frac{1}{n}} 
= \liminf_{n\to\infty} \left( [R_{11}^{n-1} R_{12}]_{i,k} \right)^{\frac{1}{n}} 
\ge \liminf_{n\to\infty} \left( [R_{11}^{n-1}]_{i,i} [R_{12}]_{i,k} \right)^{\frac{1}{n}} 
= e^{-\bar{\theta}}.
\end{align}
Hence, we obtain equation (\ref{eq:cpRlimit}). It is obvious by expression (\ref{eq:R11R12}) that, for $i\in S_2$ and $j\in\mathbb{Z}_+$, $[R^n]_{i,j}=0$. 
\end{proof}

\end{document}